\documentclass{amsart}
\usepackage{amscd,thmdefs,amsmath,amssymb,latexsym, amsfonts,txfonts} 
\usepackage{textcomp}
\usepackage{color}
\usepackage{hyperref}
\usepackage{tikz,pgfplots,pgf} 
\usepackage{enumerate}
\usepackage[all]{xy}
\tikzset{node distance=2cm, auto}
\usetikzlibrary{calc} 
\usetikzlibrary{trees} 

\def\N{\mathbb{N}}
\def\R{\mathbb{R}}

\def\K{\mathbb{K}}

\def\F{\mathcal{F}}

\def\L{\mathcal{L}}

\def\lin{\mathrm{lin}}
\def\Lip{\mathrm{Lip}}
\def\Lipo{\mathrm{Lip}_0}
\def\LipoF{\mathrm{Lip}_{0F}}

\def\Lipa{\mathrm{Lip}_\alpha}
\def\Lipp{\mathrm{Lip}_\pi}

\def\Lipdp{\mathrm{Lip}_{d_p}}
\def\Lipdu{\mathrm{Lip}_{d_1}}
\def\Lipdi{\mathrm{Lip}_{d_\infty}}
\def\Lipwp{\mathrm{Lip}_{w_p}}

\def\rank{\mathrm{rank}}
\def\codim{\mathrm{codim}}

\def\fin{\mathrm{FIN}}
\def\cofin{\mathrm{COFIN}}
\def\mfin{\mathrm{MFIN}}

\newcommand{\pair}[2]{{\langle #1, #2 \rangle}}
\newcommand{\n}[1]{ \left\|#1\right\| }

\begin{document}
\title[Duality for ideals of Lipschitz maps]{Duality for ideals of Lipschitz maps}

\author{M. G. Cabrera-Padilla}
\address{Departamento de Matem\'{a}ticas, Universidad de Almer\'{i}a, 04120 Almer\'{i}a, Spain}
\email{m\_gador@hotmail.com}

\author{J. A. Ch\'avez-Dom\'{\i}nguez}
\address{Department of Mathematics,
University of Texas at Austin,
Austin, TX 78712-1202,
and
Instituto de Ciencias Matem\'aticas,
CSIC-UAM-UC3M-UCM,
 28049, Madrid, Spain.
}
\email{jachavezd@math.utexas.edu}

\thanks{The second author's research was partially supported by NSF grant DMS-1400588 and ICMAT Severo Ochoa Grant SEV-2011-0087 (Spain), and the third author's one by Junta of Andaluc\'{i}a grant FQM-194 and Ministerio de Econom\'{i}a y Competitividad project MTM2014-58984-P}

\author{A. Jim\'{e}nez-Vargas}
\address{Departamento de Matem\'{a}ticas, Universidad de Almer\'{i}a, 04120 Almer\'{i}a, Spain}
\email{ajimenez@ual.es}
\date{\today}

\author{M.  Villegas-Vallecillos}
\address{Departamento de Matem\'{a}ticas, Universidad de C\'{a}diz, 11510 Puerto Real, Spain}
\email{moises.villegas@uca.es}



\begin{abstract}
We develop a systematic approach to the study of duality for ideals of Lipschitz maps from a metric space to a Banach space, inspired by the classical theory that relates ideals of operators and tensor norms for Banach spaces, by using the Lipschitz tensor products previously introduced by the same authors.
We first study spaces of Lipschitz maps, from a metric space to a dual Banach space, that can be represented canonically as the dual of a Lipschitz tensor product endowed with a Lipschitz cross-norm.
We show that several known examples of ideals of Lipschitz maps (Lipschitz maps, Lipschitz $p$-summing maps and maps admitting a Lipschitz factorization through a subset of an $L_p$ space) admit such a representation,
and more generally we characterize when a space of Lipschitz maps from a metric space to a dual Banach space is in canonical duality with a Lipschitz cross-norm.
Furthermore, we give conditions on the Lipschitz cross-norm that are almost equivalent to the space of Lipschitz maps having an ideal property.
We introduce a concept of operators which are approximable with respect to one of these ideals of Lipschitz maps, and identify them in terms of tensor-product notions.
Finally, we also prove a Lipschitz version of the representation theorem for maximal operator ideals. This allows us to relate Lipschitz cross-norms to ideals of Lipschitz maps taking values in general Banach spaces, and not just dual spaces.
\end{abstract}
\maketitle

\section*{Introduction}

The study of ideals of linear operators between Banach spaces, that is, families of operators that are closed under composition, has been an important tool in the study of Banach spaces. A stellar example is that of $p$-summing operators, as attested to by the astonishing number of results and applications that can be found, for example, in  \cite{Diestel-Jarchow-Tonge}.
In recent years, a number of ideals of Lipschitz maps (which in particular, are generally nonlinear) inspired by well-known and very useful ideals of linear operators between Banach spaces have appeared in the literature.
One example is the notion of Lipschitz $p$-summing operators between metric spaces, a nonlinear generalization of $p$-summing operators, which was introduced by J. D. Farmer and W. B. Johnson in \cite{fj}.
Other examples of such ideals of Lipschitz maps are operators that admit a Lipschitz factorization through an $L_p$ space \cite{Johnson-Maurey-Schechtman}, Lipschitz $p$-nuclear and Lipschitz $p$-integral operators \cite{cz}, or operators admitting a Lipschitz factorization through a subset of a Hilbert space \cite{CD-Lipschitz-factorization}.
If we restrict our attention to one of these ideals of maps from a fixed metric space to a fixed normed space, the resulting space of maps is itself a normed space.
Therefore, being able to identify the dual of such a space of maps would be interesting and useful.
That is precisely one of the questions raised by Farmer and Johnson at the end of their paper for the specific case of Lipschitz $p$-summing maps \cite{fj}.

Our purpose in this paper is to develop a systematic approach to the duality theory for ideals of Lipschitz maps from a metric space to a Banach space, a generalization of the aforementioned question of Farmer and Johnson. 
This approach is inspired on one hand by the deep and useful connections between the theories of operator ideals and of tensor norms for Banach spaces (in the spirit of the book of A. Defant and K. Floret \cite{Defant-Floret}), and on the other by the solution by the second-named author to the problem of duality for Lipschitz $p$-summing operators \cite{cd11}.
The key idea in \cite{cd11} is the concept of spaces of Banach-space-valued molecules, which plays the role of a sort of ``tensor product'' between a metric space and a Banach space. In \cite{cd11} those spaces of molecules are endowed with certain Lipschitz versions of the tensor norms of S. Chevet \cite{c} and P. Saphar \cite{s}, which are in canonical duality with spaces of  Lipschitz $p$-summing maps. It is worth noting that the use of tensor-product techniques in the research on the duality theory for Lipschitz function spaces implicitly appears already in the works \cite{j70,j71} by J. A. Johnson in the 70's.
In  \cite{ccjv}, we have formalized the notion of a Lipschitz tensor product between a metric space and a normed space, and studied its basic properties.
In the present paper we develop the duality theory that relates Lipschitz tensor products and ideals of Lipschitz maps,
by answering the following two questions:
Given an ideal of Lipschitz maps, when can it be canonically identified with the dual of a Lipschitz tensor product?
Given a Lipschitz tensor product, when can its dual space be canonically identified with an ideal of Lipschitz maps?


Let us now describe the contents of this paper.
Section \ref{s1} gathers some preliminary results on the Lipschitz tensor product $X\boxtimes E$
between a metric space $X$ and a normed space $E$, proved in the paper \cite{ccjv}.
In Section \ref{alo} we introduce and study the space $\Lipa(X,E^*)$ of $E^*$-valued $\alpha$-Lipschitz operators defined on $X$, that is, operators from $X$ to $E^*$ that induce a continuous functional on a given Lipschitz tensor product with a cross-norm norm $\alpha$ (denoted by $X \boxtimes_\alpha E$).
The $\alpha$-Lipschitz operators are in fact Lipschitz maps, which justifies the terminology.
Moreover, we show that several known examples of ideals of Lipschitz maps
--- namely Lipschitz maps, Lipschitz $p$-summing maps and maps admitting a Lipschitz factorization through a subset of an $L_p$ space --- are associated to Lipschitz cross-norms in this way.

Section \ref{alo2} addresses the duality theory for $\alpha$-Lipschitz operators and contains the main result of this paper: the space of $E^*$-valued $\alpha$-Lipschitz operators defined on $X$ is canonically isometrically isomorphic to the dual of the Lipschitz tensor product $X \boxtimes_\alpha E$.
This canonical identification is the basis of our study of the duality for ideals of Lipschitz maps. 
The section is completed by studying the several topologies on the space of $E^*$-valued $\alpha$-Lipschitz operators defined on $X$.
Thus, the main questions we are pursuing in this paper can be rephrased as:
When is the space of $\alpha$-Lipschitz operators an ideal?
Given an ideal of Lipschitz maps, when can it be represented as a space of $\alpha$-Lipschitz maps? 

In Section \ref{frlo} we take a small detour from the main theme of the paper to work out several results dealing with approximations.
We show that under minimal assumptions on the cross-norm $\alpha$, the simplest (that is,  the so-called Lipschitz finite-rank) Lipschitz operators from $X$ to $E^*$ are all $\alpha$-Lipschitz. 
We also study the Lipschitz operators from $X$ into $E^*$ that are limits in the $\alpha$-Lipschitz norm of sequences of Lipschitz finite-rank operators, which are called $\alpha$-Lipschitz approximable operators. 

In Section \ref{lobi} we formalize the notion of ideals of Lipschitz maps, which we have called Banach ideals of Lipschitz operators. We introduce these ideals and give some sufficient conditions and other necessary ones on $\Lipa(X,E^*)$ and the space of $\alpha$-Lipschitz approximable operators to be such a Banach ideal of Lipschitz operators.

In Section \ref{lobs}, we look at spaces of maps from $X$ to $E^*$ which are not necessarily ideals but nevertheless are in duality with Lipschitz cross-norms. We introduce the concept of Lipschitz operator Banach space, and give simple conditions on $\alpha$ that characterize when $\Lipa(X,E^*)$ is one such space.
As was already mentioned, it is proved in Section \ref{alo} that if $\alpha$ is a Lipschitz cross-norm on $X\boxtimes E$, then $\Lipa(X,E^*)$ can be identified with the dual of the space $X\boxtimes_\alpha E$.
We now prove a converse result,
characterizing those Lipschitz operator Banach spaces that are canonically isometrically isomorphic to the dual of $X\boxtimes_\alpha E$ for some Lipschitz cross-norm $\alpha$ on $X\boxtimes E$ (in terms of the compactness of their unit balls with respect to one of the topologies introduced in Section \ref{alo2}).

Up to this point in the paper, we have only dealt with spaces of maps from a metric space to a dual Banach space.
In Section \ref{rtmlobi}, we study the relationship between Lipschitz cross-norms and Lipschitz operator Banach ideals even in the case when the latter take values in a Banach space which is not a dual space.
The main result of this section is a Lipschitz version of the representation theorem for maximal operator ideals \cite[17.5]{Defant-Floret}, and we deduce some consequences of it including a general theorem that characterizes certain nonlinear maps by linear means (Theorem \ref{thm-linear-characterization-of-nonlinear}).

\section{Notation and preliminary results}\label{s1}

Given two metric spaces $(X,d_X)$ and $(Y,d_Y)$, let us recall that a map $f\colon X\to Y$ is said to be \emph{Lipschitz} if there exists a real constant $C\geq 0$ such that $d_Y(f(x),f(y))\leq Cd_X(x,y)$ for all $x,y\in X$. The least constant $C$ for which the preceding inequality holds will be denoted by $\Lip(f)$, that is, 
$$
\Lip(f)=\sup\left\{\frac{d_Y(f(x),f(y))}{d_X(x,y)}\colon x,y\in X, \; x\neq y\right\}.
$$ 
A pointed metric space $X$ is a metric space with a base point in $X$ , that is, a designated special point, which we will always denote by $0$. As usual, $\K$ denotes the field of real or complex numbers. We will consider a normed space $E$ over $\K$ as a pointed metric space with the distance defined by its norm and the zero vector as the base point. As is customary, $B_E$ and $S_E$ stand for the closed unit ball of $E$ and the unit sphere of $E$, respectively. 

Given two pointed metric spaces $X$ and $Y$, we denote by $\Lipo(X,Y)$ the set of all base-point preserving Lipschitz maps from $X$ to $Y$. If $E$ is a Banach space, then $\Lipo(X,E)$ is a Banach space under the Lipschitz norm given by 
$$
\Lip(f)=\sup\left\{\frac{\left\|f(x)-f(y)\right\|}{d(x,y)}\colon x,y\in X, \; x\neq y\right\}.
$$ 
The elements of $\Lipo(X,E)$ are known as \emph{Lipschitz operators}. The space $\Lipo(X,\K)$ is called the \emph{Lipschitz dual of} $X$ and it will be denoted by $X^\#$. 

For two vector spaces $E$ and $F$, $L(E,F)$ stands for the vector space of all linear operators from $E$ into $F$. 
In the case that $E$ and $F$ are Banach spaces, $\L(E,F)$ represents the Banach space of all bounded linear operators from $E$ to $F$ endowed with the canonical norm of operators. In particular, the algebraic dual $L(E,\K)$ and the topological dual $\L(E,\K)$ are denoted by $E'$ and $E^*$, respectively. For each $e\in E$ and $e^*\in E'$, we frequently will write $\langle e^*,e\rangle$ instead of $e^*(e)$. 

Throughout this paper, unless otherwise stated, $X$ will denote a pointed metric space with base point $0$ and $E$ a Banach space.

We now some concepts and facts whose proofs can be found in \cite{ccjv}. Let $X$ be a pointed metric space and let $E$ be a Banach space. The \emph{Lipschitz tensor product} $X\boxtimes E$ is the linear span of all linear functionals $\delta_{(x,y)}\boxtimes e$ on $\Lipo(X,E^*)$ of the form 
$$
\left(\delta_{(x,y)}\boxtimes e\right)(f)=\left\langle f(x)-f(y),e\right\rangle
$$
for $(x,y)\in X^2$ and $e\in E$. A norm $\alpha$ on $X\boxtimes E$ is a \emph{Lipschitz cross-norm} if 
$$
\alpha\left(\delta_{(x,y)}\boxtimes e\right)=d(x,y)\left\|e\right\|
$$
for all $(x,y)\in X^2$ and $e\in E$. We denote by $X\boxtimes_\alpha E$ the linear space $X\boxtimes E$ with norm $\alpha$, and by $X\widehat{\boxtimes}_\alpha E$ the completion of $X\boxtimes_\alpha E$. A Lipschitz cross-norm $\alpha$ on $X\boxtimes E$ is called \emph{dualizable} if given $g\in X^\#$ and $\phi\in E^*$, we have
$$
\left|\sum_{i=1}^n\left(g(x_i)-g(y_i)\right)\left\langle \phi,e_i\right\rangle\right|
\leq\Lip(g)\left\|\phi\right\|\alpha\left(\sum_{i=1}^n \delta_{(x_i,y_i)}\boxtimes e_i\right)
$$
for all $\sum_{i=1}^n \delta_{(x_i,y_i)}\boxtimes e_i\in X\boxtimes E$, and it is called \emph{uniform} if given $h\in\Lipo(X,X)$ and $T\in\L(E,E)$, we have
$$
\alpha\left(\sum_{i=1}^n\delta_{(h(x_i),h(y_i))}\boxtimes T(e_i)\right)
\leq\Lip(h)\left\|T\right\|\alpha\left(\sum_{i=1}^n \delta_{(x_i,y_i)}\boxtimes e_i\right)
$$ 
for all $\sum_{i=1}^n \delta_{(x_i,y_i)}\boxtimes e_i\in X\boxtimes E$.

For each $\sum_{i=1}^n \delta_{(x_i,y_i)}\boxtimes e_i\in X\boxtimes E$, the \emph{Lipschitz injective norm} on $X\boxtimes E$ is defined by 
$$
\varepsilon\left(\sum_{i=1}^n \delta_{(x_i,y_i)}\boxtimes e_i\right)=\sup\left\{\left|\sum_{i=1}^n\left(g(x_i)-g(y_i)\right)\left\langle\phi,e_i\right\rangle\right|\colon g\in B_{X^\#}, \; \phi\in B_{E^*}\right\}.
$$
For each $u\in X\boxtimes E$, the \emph{Lipschitz projective norm} $\pi$ and the \emph{Lipschitz $p$-nuclear} norm $d_p$ for $1\leq p\leq\infty$ are defined on $X\boxtimes E$ as 
\begin{align*}
\pi(u)&=\inf\left\{\sum_{i=1}^n d(x_i,y_i)\left\|e_i\right\|\colon u=\sum_{i=1}^n\delta_{(x_i,y_i)}\boxtimes e_i\right\},\\
d_p(u)&=\inf\left\{\left(\sup_{g\in B_{X^\#}}\left(\sum_{i=1}^n\left|g(x_i)-g(y_i)\right|^p\right)^{\frac{1}{p'}}\right)\left(\sum_{i=1}^n\left\|e_i\right\|^p\right)^{\frac{1}{p}}\colon u=\sum_{i=1}^n \delta_{(x_i,y_i)}\boxtimes e_i\right\},
\end{align*}
where the infimum is taken over all such representations of $u$. It is known that $\varepsilon$, $\pi$ and $d_p$ for $p\in [1,\infty]$ are uniform and dualizable Lipschitz cross-norms on $X\boxtimes E$ and $d_1=\pi$. Moreover, $\varepsilon$ is the least dualizable Lipschitz cross-norm on $X\boxtimes E$ and $\pi$ is the greatest Lipschitz cross-norm on $X\boxtimes E$. In fact, a norm $\alpha$ on $X\boxtimes E$ is a dualizable Lipschitz cross-norm if and only if $\varepsilon\leq\alpha\leq\pi$. 

If $g\in X^\#$ and $\phi\in E^*$, we can consider the linear functional $g\boxtimes\phi$ on $X\boxtimes E$ defined by 
$$
(g\boxtimes\phi)\left(\sum_{i=1}^n \delta_{(x_i,y_i)}\boxtimes e_i\right)=\sum_{i=1}^n\left(g(x_i)-g(y_i)\right)\left\langle \phi,e_i\right\rangle .
$$
The \emph{associated Lipschitz tensor product} of $X\boxtimes E$, denoted by $X^\#\boxast E^*$, is the linear span of all linear functionals $g\boxtimes\phi$ on $X\boxtimes E$ for $g\in X^\#$ and $\phi\in E^*$. A norm $\beta$ on $X^\#\boxast E^*$ is called a \emph{Lipschitz cross-norm} if 
$$
\beta(g\boxtimes\phi)=\Lip(g)\left\|\phi\right\|
$$
for all $g\in X^\#$ and $\phi\in E^*$. Denote by $X^\#\boxast_\beta E^*$ the linear space $X^\#\boxast E^*$ with norm $\beta$, and by $X^\#\widetilde{\boxast}_\beta E^*$ the completion of $X^\#\boxast_\beta E^*$.

Given a dualizable Lipschitz cross-norm $\alpha$ on $X\boxtimes E$, the map $\alpha'\colon X^\#\boxast E^*\to\R$, given by   
$$
\alpha'\left(\sum_{j=1}^m g_j\boxtimes \phi_j\right)
=\sup\left\{\left|\left(\sum_{j=1}^m g_j\boxtimes\phi_j\right)\left(\sum_{i=1}^n \delta_{(x_i,y_i)}\boxtimes e_i\right)\right|
\colon \alpha\left(\sum_{i=1}^n \delta_{(x_i,y_i)}\boxtimes e_i\right)\leq 1\right\}
$$
for $\sum_{j=1}^m g_j\boxtimes \phi_j\in X^\#\boxast E^*$, is a Lipschitz cross-norm on $X^\#\boxast E^*$ called the \emph{associated Lipschitz norm of $\alpha$}, and clearly $X^\#\boxast_{\alpha'} E^*$ is a normed linear subspace of $(X\widetilde{\boxtimes}_\alpha E)^*$.

For $h\in\Lipo(X,Y)$ and $T\in\L(E,F)$, we also consider the linear operator $h\boxtimes T$ from $X\boxtimes E$ to $Y\boxtimes F$ given by 
$$
(h\boxtimes T)\left(\sum_{i=1}^n\delta_{(x_i,y_i)}\boxtimes e_i\right)
=\sum_{i=1}^n\delta_{(h(x_i),h(y_i))}\boxtimes T(e_i).
$$

\section{Cross-norm-Lipschitz operators}\label{alo}

In this section we introduce and study the concept that will give rise to the canonical association between Lipschitz cross-norms and ideals of Lipschitz maps.
It is the concept of cross-norm-Lipschitz operator from $X$ to $E^*$, which is an operator that induces a bounded functional on $X \boxtimes E$ endowed with a Lipschitz cross-norm. To be precise:

\begin{definition}\label{def-cross-norm-Lipschitz operator}
Let $\alpha$ be a Lipschitz cross-norm on $X\boxtimes E$. A base-point preserving map $f\colon X\to E^*$ is said to be an $\alpha$-Lipschitz operator if there exists a real constant $C\geq 0$ such that
$$
\left|\sum_{i=1}^n\left\langle f(x_i)-f(y_i),e_i \right\rangle\right|
\leq C\alpha\left(\sum_{i=1}^n\delta_{(x_i,y_i)}\boxtimes e_i\right)
$$
for all $\sum_{i=1}^n\delta_{(x_i,y_i)}\boxtimes e_i\in X\boxtimes E$. The infimum of such constants $C$ is denoted by $\Lipa(f)$ and called the $\alpha$\emph{-Lipschitz norm of} $f$. The set of all $\alpha$-Lipschitz operators from $X$ into $E^*$ is denoted by $\Lipa(X,E^*)$. 
\end{definition}

The following lemma justifies the terminology used in Definition \ref{def-cross-norm-Lipschitz operator}, since every $\alpha$-Lipschitz operator turns out to be a Lipschitz operator.

\begin{lemma}\label{lem-alpha-Lipschitz}
Let $\alpha$ be a Lipschitz cross-norm on $X\boxtimes E$. Then every $\alpha$-Lipschitz operator $f\colon X\to E^*$ is Lipschitz and $\Lip(f)\leq\Lipa(f)$.
\end{lemma}

\begin{proof}
Let $f\in\Lipa(X,E^*)$. 
For $x,y\in X$ and $e\in E$, we have
$$
\left|\left\langle f(x)-f(y),e\right\rangle\right|
\leq\Lipa(f)\alpha\left(\delta_{(x,y)}\boxtimes e\right)
=\Lipa(f)d(x,y)\left\|e\right\|,
$$
hence $||f(x)-f(y)||\leq\Lipa(f)d(x,y)$ and so $f\in\Lipo(X,E^*)$ with $\Lip(f)\leq\Lipa(f)$.
\end{proof}

\begin{remark}\label{remark-def-cross-norm-Lipschitz operator}
Note that if $u=\sum_{i=1}^n\delta_{(x_i,y_i)}\boxtimes e_i\in X\boxtimes E$ and $f\in\Lipo(X,E^*)$, then  
$$
u(f)=\sum_{i=1}^n\left\langle f(x_i)-f(y_i),e_i\right\rangle,
$$
and therefore $f$ is in $\Lipa(X,E^*)$ if and only if $\left|u(f)\right|\leq C\alpha(u)$ for all $u\in X\boxtimes E$. Moreover, 
\begin{align*}
\Lipa(f)&=\min\left\{C\geq 0\colon \left|u(f)\right|\leq C\alpha(u),\; \forall u\in X\boxtimes E \right\}\\
        &=\sup\left\{\left|u(f)\right|\colon u\in X\boxtimes E,\; \alpha(u)\leq 1\right\}\\
        &=\sup\left\{\left|u(f)\right|\colon u\in X\boxtimes E,\; \alpha(u)=1\right\}.
\end{align*}
\end{remark}

\begin{lemma}\label{lem-Lipa-norm-space}
Let $\alpha$ be a Lipschitz cross-norm on $X\boxtimes E$. Then $\Lipa(X,E^*)$ is a normed space with the $\alpha$-Lipschitz norm. 
\end{lemma}

\begin{proof}
Let $f,g\in\Lipa(X,E^*)$ and $\lambda\in\K$. Clearly, $\Lipa(f)\geq 0$. Assume $f\neq 0$. Then, for some $x\in X$ and $e\in E$, $\langle f(x),e\rangle\neq 0$, that is, $\langle\delta_{(x,0)}\boxtimes e,f\rangle\neq 0$. This implies that $\delta_{(x,0)}\boxtimes e\neq 0$ and thus $\alpha(\delta_{(x,0)}\boxtimes e)>0$. Then we have 
$\Lipa(f)\geq |\langle f(x),e\rangle|/\alpha(\delta_{(x,0)}\boxtimes e)>0$, as required. 
Next we use Remark \ref{remark-def-cross-norm-Lipschitz operator}. For any $u\in X\boxtimes E$, we obtain
$$
\left|u(\lambda f)\right|
=\left|\lambda u(f)\right|
=\left|\lambda\right|\left|u(f)\right|
\leq\left|\lambda\right|\Lipa(f)\alpha(u),
$$
and therefore $\lambda f\in\Lipa(X,E^*)$ and $\Lipa(\lambda f)\leq\left|\lambda\right|\Lipa(f)$. Moreover, by above-proved, if $\lambda=0$, then $\Lipa(\lambda f)=0=\left|\lambda\right|\Lipa(f)$, and if $\lambda\neq 0$, we have $\Lipa(f)=\Lipa(\lambda^{-1}(\lambda f))\leq\left|\lambda\right|^{-1}\Lipa(\lambda f)$, and hence $\left|\lambda\right|\Lipa(f)\leq\Lipa(\lambda f)$. This proves that $\Lipa(\lambda f)=\left|\lambda\right|\Lipa(f)$. Finally, for all $u\in X\boxtimes E$, 
$$
\left|u(f+g)\right|=\left|u(f)+u(g)\right|\leq \left|u(f)\right|+\left|u(g)\right|\leq(\Lipa(f)+\Lipa(g))\alpha(u),
$$
and so $f+g\in\Lipa(X,E^*)$ and $\Lipa(f+g)\leq\Lipa(f)+\Lipa(g)$. This completes the proof of the lemma.
\end{proof}

We now identify the space of all Lipschitz operators from $X$ into $E^*$ with the space of all $\pi$-Lipschitz operators.

\begin{lemma}\label{lemma-Lipschitz-pi}
The sets $\Lipo(X,E^*)$ and $\Lipp(X,E^*)$ are equal. Moreover, $\Lip(f)=\Lipp(f)$ for all $f\in\Lipo(X,E^*)$.
\end{lemma}

\begin{proof}
Let $f\in\Lipo(X,E^*)$. Since $|u(f)|\leq\Lip(f)\pi(u)$ for all $u\in X\boxtimes E$, we infer that $f\in\Lipp(X,E^*)$ and $\Lipp(f)\leq\Lip(f)$. The lemma now follows by Lemma \ref{lem-alpha-Lipschitz}.
\end{proof}

J. D. Farmer and W. B. Johnson introduced in \cite{fj} the notion of Lipschitz $p$-summing operators between metric spaces for $1\leq p<\infty$. See \cite{cd11} for the case $p=\infty$. Let us recall that if $X$ and $Y$ are pointed metric spaces, a map $f\in\Lipo(X,Y)$ is said to be \emph{Lipschitz $p$-summing} ($1\leq p\leq\infty$) if there exists a constant $C\geq 0$ such that regardless of the natural number $n$ and regardless of the choices of points $x_1,\ldots,x_n\in X$ and $y_1,\ldots,y_n\in X$, we have the inequality
\begin{align*}
\left(\sum_{i=1}^n d(f(x_i),f(y_i))^p\right)^{\frac{1}{p}}\leq C \sup_{g\in B_{X^\#}}\left(\sum_{i=1}^n\left|g(x_i)-g(y_i)\right|^p\right)^{\frac{1}{p}}\quad &\text{if} \quad 1\leq p<\infty,\\
\max_{1\leq i\leq n}d(f(x_i),f(y_i))\leq C \sup_{g\in B_{X^\#}}\left(\max_{1\leq i\leq n}\left|g(x_i)-g(y_i)\right|\right)\quad &\text{if} \quad  p=\infty.
\end{align*}
The infimum of such constants is denoted by $\pi_p^L(f)$ and called the \emph{Lipschitz $p$-summing norm of} $f$.

If $E$ is a Banach space, the set $\Pi^L_p(X,E^*)$ of all Lipschitz $p$-summing operators from $X$ into $E^*$ with the norm $\pi_p^L$ is a Banach space (see \cite{fj,cd11}). If $p'$ is the conjugate index of $p\in [1,\infty]$, we next identify the Lipschitz $p$-summing operators from $X$ to $E^*$ with the $d_{p'}$-Lipschitz operators.

\begin{theorem}\label{teo-alpha-Lipschitz-summing}
Let $1\leq p\leq\infty$. Then $\Lipdp(X,E^*)=\Pi^L_{p'}(X,E^*)$ and $\Lipdp(f)=\pi_{p'}^L(f)$ for every $f\in\Lipdp(X,E^*)$.
\end{theorem}

\begin{proof}
Let $f\in\Pi^L_{p'}(X,E^*)$ and $u\in X\boxtimes E$. If $\sum_{i=1}^n \delta_{(x_i,y_i)}\boxtimes e_i$ is a representation of $u$, then 
\begin{align*}
\left|u(f)\right|
&=\left|\sum_{i=1}^n \left\langle f(x_i)-f(y_i),e_i\right\rangle\right|\\
&\leq\sum_{i=1}^n\left\|f(x_i)-f(y_i)\right\|\left\|e_i\right\|\\
&\leq\left(\sum_{i=1}^n\left\|f(x_i)-f(y_i)\right\|^{p'}\right)^{\frac{1}{p'}}
\left(\sum_{i=1}^n\left\|e_i\right\|^p\right)^{\frac{1}{p}}\\
&\leq\pi_{p'}^L(f)\left(\sum_{i=1}^n\left\|e_i\right\|^p\right)^{\frac{1}{p}}
\sup_{g\in B_{X^\#}}\left(\sum_{i=1}^n\left|g(x_i)-g(y_i)\right|^{p'}\right)^{\frac{1}{p'}}
\end{align*}
in the case $1<p<\infty$.  When $p=1$, we have 
\begin{align*}
\left|u(f)\right|
&\leq\sum_{i=1}^n\left\|f(x_i)-f(y_i)\right\|\left\|e_i\right\|\\
&\leq\left(\max_{1\leq i\leq n}\left\|f(x_i)-f(y_i)\right\|\right)\sum_{i=1}^n\left\|e_i\right\|\\
&\leq\pi_{\infty}^L(f)\sup_{g\in B_{X^\#}}\left(\max_{1\leq i\leq n}\left|g(x_i)-g(y_i)\right|\right)\sum_{i=1}^n\left\|e_i\right\|,
\end{align*}
and, for $p=\infty$, 
\begin{align*}
\left|u(f)\right|
&\leq\sum_{i=1}^n\left\|f(x_i)-f(y_i)\right\|\left\|e_i\right\|\\
&\leq\left(\max_{1\leq i\leq n}\left\|e_i\right\|\right)\sum_{i=1}^n\left\|f(x_i)-f(y_i)\right\|\\
&\leq\pi_1^L(f)\left(\max_{1\leq i\leq n}\left\|e_i\right\|\right)\sup_{g\in B_{X^\#}}\left(\sum_{i=1}^n\left|g(x_i)-g(y_i)\right|\right).
\end{align*}
Taking the infimum over all such representations of $u$, we deduce that $\left|u(f)\right|\leq\pi_{p'}^L(f)d_p(u)$. Since $u$ was arbitrary in $X\boxtimes E$, it follows that $f\in\Lipdp(X,E^*)$ and $\Lipdp(f)\leq\pi_{p'}^L(f)$.




Conversely, let $f\in\Lipdp(X,E^*)$ and let $n\in\N$, $x_1,\ldots,x_n\in X$ and $y_1,\ldots,y_n\in X$. Let $\varepsilon>0$. Then, for each $i\in\{1,\ldots,n\}$, there exists $e_i\in E$ with $||e_i||\leq 1+\varepsilon$ such that $\langle f(x_i)-f(y_i),e_i\rangle=||f(x_i)-f(y_i)||$. It is elementary that the map $T\colon\K^n\to\K$, defined by 
$$
T(t_1,\ldots,t_n)=\sum_{i=1}^n t_i \left\|f(x_i)-f(y_i)\right\|,\quad\forall (t_1,\ldots,t_n)\in\K^n,
$$
is linear and continuous on $(\K^n,||\cdot||_p)$ with 
$$
\left\|T\right\|
=\left\{
\begin{array}{lll}
	\left(\sum_{i=1}^n\left\|f(x_i)-f(y_i)\right\|^{p'}\right)^{1/p'}&\text{if}&1<p\leq\infty,\\
& &\\
\max_{1\leq i\leq n}\left\|f(x_i)-f(y_i)\right\|&\text{if}&p=1.
\end{array}
\right.
$$
For any $(t_1,\ldots,t_n)\in\K^n$ with $||(t_1,\ldots,t_n)||_p\leq 1$, we have
$$
\left|T(t_1,\ldots,t_n)\right|
=\left|\sum_{i=1}^n \left\langle f(x_i)-f(y_i),t_ie_i\right\rangle\right|
=\left|\sum_{i=1}^n \left\langle \delta_{(x_i,y_i)}\boxtimes (t_ie_i),f\right\rangle\right|
\leq\Lipdp(f)d_p\left(\sum_{i=1}^n\delta_{(x_i,y_i)}\boxtimes (t_ie_i)\right).
$$
If $1<p<\infty$, it follows that 
\begin{align*}
\left|T(t_1,\ldots,t_n)\right|
&\leq\Lipdp(f)\left(\sum_{i=1}^n\left\|t_ie_i\right\|^p\right)^{\frac{1}{p}}\sup_{g\in B_{X^\#}}\left(\sum_{i=1}^n\left|g(x_i)-g(y_i)\right|^{p'}\right)^{\frac{1}{p'}}\\
&\leq\Lipdp(f)(1+\varepsilon)\sup_{g\in B_{X^\#}}\left(\sum_{i=1}^n\left|g(x_i)-g(y_i)\right|^{p'}\right)^{\frac{1}{p'}},
\end{align*}
consequently, we have
$$
\left(\sum_{i=1}^n\left\|f(x_i)-f(y_i)\right\|^{p'}\right)^{\frac{1}{p'}}
\leq\Lipdp(f)(1+\varepsilon)\sup_{g\in B_{X^\#}}\left(\sum_{i=1}^n\left|g(x_i)-g(y_i)\right|^{p'}\right)^{\frac{1}{p'}},
$$
and since $\varepsilon$ was arbitrary, we deduce that 
$$
\left(\sum_{i=1}^n\left\|f(x_i)-f(y_i)\right\|^{p'}\right)^{\frac{1}{p'}}
\leq\Lipdp(f)\sup_{g\in B_{X^\#}}\left(\sum_{i=1}^n\left|g(x_i)-g(y_i)\right|^{p'}\right)^{\frac{1}{p'}},
$$
and so $f\in\Pi_{p'}^L(X,E^*)$ with $\pi_{p'}^L(f)\leq\Lipdp(f)$.  Reasoning similarly, we arrive at the same conclusion for the cases $p=1$ and $p=\infty$. Indeed, if $p=1$, we have  
\begin{align*}
\left|T(t_1,\ldots,t_n)\right|
&\leq\Lipdu(f)\left(\sum_{i=1}^n\left\|t_ie_i\right\|\right)\sup_{g\in B_{X^\#}}\left(\max_{1\leq i\leq n}\left|g(x_i)-g(y_i)\right|\right)\\
&\leq\Lipdu(f)(1+\varepsilon)\sup_{g\in B_{X^\#}}\left(\max_{1\leq i\leq n}\left|g(x_i)-g(y_i)\right|\right),
\end{align*}
which gives    
$$
\max_{1\leq i\leq n}\left\|f(x_i)-f(y_i)\right\|\leq\Lipdu(f)\sup_{g\in B_{X^\#}}\left(\max_{1\leq i\leq n}\left|g(x_i)-g(y_i)\right|\right),
$$
and so $f\in\Pi_{\infty}^L(X,E^*)$ with $\pi_{\infty}^L(f)\leq\Lipdu(f)$.  For $p=\infty$, we have 
\begin{align*}
\left|T(t_1,\ldots,t_n)\right|
&\leq\Lipdi(f)\left(\max_{1\leq i\leq n}\left\|t_ie_i\right\|\right)\sup_{g\in B_{X^\#}}\left(\sum_{i=1}^n\left|g(x_i)-g(y_i)\right|\right)\\
&\leq\Lipdi(f)(1+\varepsilon)\sup_{g\in B_{X^\#}}\left(\sum_{i=1}^n\left|g(x_i)-g(y_i)\right|\right),
\end{align*}
hence 
$$
\sum_{i=1}^n\left\|f(x_i)-f(y_i)\right\|\leq\Lipdi(f)\sup_{g\in B_{X^\#}}\left(\sum_{i=1}^n\left|g(x_i)-g(y_i)\right|\right),
$$
and so $f\in\Pi_{1}^L(X,E^*)$ with $\pi_{1}^L(f)\leq\Lipdi(f)$. 
\end{proof}


A similar description can be obtained for the class of maps admitting a Lipschitz factorization through a subset of an $L_p$ space. This has been proved, though stated in a slightly different language, in \cite{CD-Lipschitz-factorization} for $p=2$. Let us recall the basic definitions.
For any pointed metric spaces $X$ and $Y$, $f \in \Lipo(X,Y)$ and $1 \le p \le \infty$, consider the infimum of $\Lip(R) \cdot \Lip(S)$ taken over all factorizations of the form
\begin{equation*}
	\xymatrix{
	 &Z  \ar[dr]^{S} &  \\
	X \ar[ur]^R \ar[rr]_{f}  & &Y   
	}
\end{equation*}
where $\mu$ is a measure and $Z$ is a subset of $L_p(\mu)$.
We will denote this infimum by $\gamma_p^{\Lip}(f)$, inspired by the notation of a similar situation in Banach space theory,
and by $\Gamma_p^\Lip(X,Y)$ the set of all maps in  $\Lipo(X,Y)$ admitting such a factorization.
For a pointed metric space $X$ and a Banach space $E$, it is not hard to show that $\big( \Gamma_p^\Lip(X,E) , \gamma_p^\Lip\big)$ is
a Banach space.
For $x_j,x_j',y_i,y_i'\in X, \lambda_i,\mu_j\in\R$ and $1 \le p \le \infty$, we write $(\lambda_i,y_i,y_i')_{i=1}^n \prec_p (\mu_j,x_j,x_j')_{j=1}^m$ if for every $f \in X^\#$,
$$
\sum_{i=1}^n \big|\lambda_i[f(y_i)-f(y_i')]\big|^p \le \sum_{j=1}^m \big|\mu_j[f(x_j)-f(x_j')]\big|^p.
$$
Equivalently (see \cite[Lemma 3.2]{CD-Lipschitz-factorization}), this means that there exists a linear map $A = (a_{ij})\colon\ell_p^m \to \ell_p^n$ of norm at most one such that for each
$1 \le i \le n$,
$$
\lambda_i\delta_{(y_i,y_i')} = \sum_{j=1}^m a_{ij}\mu_j\delta_{(x_j,x_j')}.
$$

\begin{definition}
Let $X$ be a pointed metric space, $E$ be a Banach space and $1 \le p \le \infty$.
For $u \in X \boxtimes E$, define
\begin{multline*}
w_p(u) = \inf \bigg\{  \Big(\sum_{i=1}^n \n{e_i}^{p} \Big)^{1/p}\Big(\sum_{j=1}^m \mu_j^{p'}d(x_j,x_j')^{p'}\Big)^{1/p'} \colon x_j,x_j',y_i,y_i'\in X, \lambda_i,\mu_j\in\R,  e_i \in E,  \\
 u = \sum_{i=1}^n \lambda_i\delta_{(y_i,y_i')} \boxtimes e_i \text{ and } (\lambda_i,y_i,y_i')_{i=1}^n \prec_{p'} (\mu_j,x_j,x_j')_{j=1}^m \bigg\}.
\end{multline*}
\end{definition}

Similar arguments to those in \cite{CD-Lipschitz-factorization} for $p=2$ show that the norm $w_p$ is a uniform and dualizable Lipschitz cross-norm on $X\boxtimes E$. Furthermore, the Banach space of all $w_p$-Lipschitz operators from $X$ to $E^*$ can be identified with the Banach space $\Gamma^{\Lip}_{p'}(X,E^*)$, according to the following restatement of \cite[Theorem 4.5]{CD-Lipschitz-factorization} (though in that paper the proof is written out in the special case $p=2$, the details carry over to the general case).

\begin{theorem}\label{teo-Lipschitz-factorization}
Let $1\leq p\leq\infty$. Then $\Lipwp(X,E^*)=\Gamma^{\Lip}_{p'}(X,E^*)$ and $\Lipwp(f)=\gamma_{p'}^{\Lip}(f)$ for every $f\in\Lipwp(X,E^*)$.
\end{theorem}


\section{Duality for spaces of cross-norm-Lipschitz operators}\label{alo2}

As was expected from the definition, we now verify that
if $\alpha$ is a Lipschitz cross-norm on $X\boxtimes E$, there is a canonical identification between the normed space $\Lipa(X,E^*)$ and the dual space of $X\widehat{\boxtimes}_\alpha E$. In particular, $\Lipa(X,E^*)$ will be a Banach space.


\begin{theorem}\label{teo-dual}
Let $\alpha$ be a Lipschitz cross-norm on $X\boxtimes E$. Then $\Lipa(X,E^*)$ is isometrically isomorphic to $(X\widehat{\boxtimes}_\alpha E)^*$, via the map $\Lambda\colon\Lipa(X,E^*)\to (X\widehat{\boxtimes}_\alpha E)^*$ defined by
$$
\Lambda(f)(u)=\sum_{i=1}^n\left\langle f(x_i)-f(y_i),e_i\right\rangle
$$
for $f\in\Lipa(X,E^*)$ and $u=\sum_{i=1}^n \delta_{(x_i,y_i)}\boxtimes e_i\in X\boxtimes E$. Its inverse $\Lambda^{-1}\colon (X\widehat{\boxtimes}_\alpha E)^*\to\Lipa(X,E^*)$ is the map given by 
$$
\left\langle \Lambda^{-1}(\varphi)(x),e\right\rangle
=\varphi(\delta_{(x,0)}\boxtimes e) 
$$
for $\varphi\in (X\widehat{\boxtimes}_\alpha E)^*$, $x\in X$ and $e\in E$.
\end{theorem}

\begin{proof}
By \cite[Corollary 1.8]{ccjv}, the map $f\mapsto\widetilde{\Lambda}(f)$ from $\Lipo(X,E^*)$ into $(X\boxtimes E)'$, defined by 
$$
\widetilde{\Lambda}(f)(u)=\sum_{i=1}^n\left\langle f(x_i)-f(y_i),e_i\right\rangle 
$$
for $f\in\Lipo(X,E^*)$ and $u=\sum_{i=1}^n \delta_{(x_i,y_i)}\boxtimes e_i\in X\boxtimes E$, is a linear monomorphism. Since $\Lipa(X,E^*)$ is a linear subspace of $\Lipo(X,E^*)$ by Lemma \ref{lem-Lipa-norm-space}, and  $\Lambda$ is the restriction of $
\widetilde{\Lambda}$ to $\Lipa(X,E^*)$, it follows that $\Lambda$ is a linear monomorphism from $\Lipa(X,E^*)$ into $(X\boxtimes E)'$. Let $f\in\Lipa(X,E^*)$. Then 
$$
\left|\Lambda(f)(u)\right|=\left|\sum_{i=1}^n\left\langle f(x_i)-f(y_i), e_i\right\rangle\right|\leq\Lipa(f)\alpha(u)
$$
for all $u=\sum_{i=1}^n \delta_{(x_i,y_i)}\boxtimes e_i\in X\boxtimes E$. This implies that $\Lambda(f)$ is bounded on $X\boxtimes_\alpha E$ and also on its completion $X\widehat{\boxtimes}_\alpha E$ by the denseness of the former set in the latter one. Therefore $\Lambda(f)\in (X\widehat{\boxtimes}_\alpha E)^*$ and $||\Lambda(f)||\leq\Lipa(f)$. In order to see that $\Lambda$ is a surjective isometry, let $\varphi$ be in $(X\widehat{\boxtimes}_\alpha E)^*$. Define $f\colon X\to E^*$ by
$$
\left\langle f(x),e\right\rangle=\varphi(\delta_{(x,0)}\boxtimes e)\qquad\left(x\in X,\; e\in E\right).
$$ 
It is easy to check that $f(x)$ is a well-defined bounded linear functional on $E$ and that $f$ is well-defined. Notice that $\langle f(x)-f(y),e\rangle=\varphi(\delta_{(x,y)}\boxtimes e)$ for all $x,y\in X$ and $e\in E$. For any $\sum_{i=1}^n\delta_{(x_i,y_i)}\boxtimes e_i\in X\boxtimes E$, we have
$$
\left|\sum_{i=1}^n\left\langle f(x_i)-f(y_i),e_i \right\rangle\right|
=\left|\varphi\left(\sum_{i=1}^n\delta_{(x_i,y_i)}\boxtimes e_i\right)\right|
\leq \left\|\varphi\right\|\alpha\left(\sum_{i=1}^n\delta_{(x_i,y_i)}\boxtimes e_i\right),
$$
and therefore $f\in\Lipa(X,E^*)$ and $\Lipa(f)\leq ||\varphi||$. For any $u=\sum_{i=1}^n \delta_{(x_i,y_i)}\boxtimes e_i\in X\boxtimes E$, we obtain 
$$
\Lambda(f)(u)
=\sum_{i=1}^n\left\langle f(x_i)-f(y_i),e_i\right\rangle
=\sum_{i=1}^n \varphi(\delta_{(x_i,y_i)}\boxtimes e_i)
=\varphi\left(\sum_{i=1}^n \delta_{(x_i,y_i)}\boxtimes e_i\right)
=\varphi(u).
$$
Hence $\Lambda(f)=\varphi$ on a dense subspace of $X\widehat{\boxtimes}_\alpha E$ and, consequently, $\Lambda(f)=\varphi$. Moreover, $\Lipa(f)\leq ||\varphi||=||\Lambda(f)||$ as required. Finally, it follows that $\langle\Lambda^{-1}(\varphi)(x),e\rangle=\langle f(x),e\rangle=\varphi(\delta_{(x,0)}\boxtimes e) $ for $\varphi\in (X\widehat{\boxtimes}_\alpha E)^*$, $x\in X$ and $e\in E$.
\end{proof}

Theorem \ref{teo-dual} and Lemma \ref{lemma-Lipschitz-pi} give the next result of J. A. Johnson \cite[Theorems 4.1 and 5.8]{j70}.

\begin{corollary}\label{01}
The space $\Lipo(X,E^*)$ is isometrically isomorphic to $(X\widehat{\boxtimes}_\pi E)^*$.
\end{corollary}

From Theorems \ref{teo-dual} and \ref{teo-alpha-Lipschitz-summing}, we derive the following description for the space of Lipschitz $p$-summing operators from $X$ into $E^*$. Compare it with \cite[Theorem 4.3]{cd11}.

\begin{corollary}
For $1\leq p\leq\infty$, the space $\Pi^L_p(X,E^*)$ is isometrically isomorphic to $(X\widehat{\boxtimes}_{d_{p'}} E)^*$.
\end{corollary}

Similarly, Theorems \ref{teo-dual} and \ref{teo-Lipschitz-factorization} give the following identification, stated in \cite[Corollary 4.6]{CD-Lipschitz-factorization} for $p=2$.

\begin{corollary}
For $1\leq p\leq\infty$, the space $\Gamma_p^{\Lip}(X,E^*)$ is isometrically isomorphic to $(X\widehat{\boxtimes}_{w_{p'}} E)^*$.
\end{corollary}

Since $\Lipa(X,E^*)$ is a dual space by Theorem \ref{teo-dual}, we may consider it equipped with its weak* topology.
  
\begin{definition}\label{def-weak-topologies}
Let $\alpha$ be a Lipschitz cross-norm on $X\boxtimes E$. The weak* topology (in short, w*) on $\Lipa(X,E^*)$ is the weak* topology on $(X\widehat{\boxtimes}_\alpha E)^*$, that is, the topology induced by the linear space $\kappa_{X\widehat{\boxtimes}_\alpha E}(X\widehat{\boxtimes}_\alpha E)$ of linear functionals on $(X\widehat{\boxtimes}_\alpha E)^*$, where $\kappa_{X\widehat{\boxtimes}_\alpha E}$ is the canonical injection from $X\widehat{\boxtimes}_\alpha E$ into $(X\widehat{\boxtimes}_\alpha E)^{**}$.
\end{definition}

We also introduce on $\Lipa(X,E^*)$ another topology that we will use later.

\begin{definition}
Let $\alpha$ be a Lipschitz cross-norm on $X\boxtimes E$. The weak* Lipschitz operator topology (in short, w*Lo) on $\Lipa(X,E^*)$ is the topology induced by the linear space $X\boxtimes E$ of linear functionals on $\Lipa(X,E^*)$.
\end{definition}

The following facts on w*Lo can be deduced from the theory on topologies induced by families of functions (see, for example, \cite[Section 2.4]{meg}). 

\begin{remark}\label{remark new}
Let $\alpha$ be a Lipschitz cross-norm on $X\boxtimes E$.
\begin{enumerate}
	\item w*Lo is a locally convex topology on $\Lipa(X,E^*)$, and the dual space of $\Lipa(X,E^*)$ with respect to this topology is $X\boxtimes E$. Since the family of functions $X\boxtimes E$ is separating, then w*Lo is completely regular.
	\item If $\{f_\gamma\}$ is a net in $\Lipa(X,E^*)$ and $f\in\Lipa(X,E^*)$, then $\{f_\gamma\}$ converges to $f$ in the w*Lo topology if and only if $\{u(f_\gamma)\}$ converges to $u(f)$ for each $u\in X\boxtimes E$.
	\item If $B(X,E^*)$ is a linear subspace of $\Lipa(X,E^*)$ and $\Lipa(X,E^*)$ is equipped with the w*Lo topology, then the relative w*Lo topology of $\Lipa(X,E^*)$ on $B(X,E^*)$ agrees with the topology induced by the linear space $\{\left.u\right|_{B(X,E^*)}\colon u\in X\boxtimes E\}$ of linear functionals on $B(X,E^*)$.
\end{enumerate}
\end{remark}

\begin{corollary}\label{cor-theo-dual}
Let $\alpha$ be a Lipschitz cross-norm on $X\boxtimes E$. 
\begin{enumerate}
\item A net $\{f_\gamma\}$ in $\Lipa(X,E^*)$ converges to $f\in\Lipa(X,E^*)$ in the weak* topology if and only if $\{u(f_\gamma)\}$ converges to $u(f)$ for every $u\in X\widehat{\boxtimes}_\alpha E$.
\item On $\Lipa(X,E^*)$, the weak* Lipschitz operator topology is weaker than the weak* topology. Moreover, on bounded subsets of $\Lipa(X,E^*)$, both topologies agree. 
\end{enumerate}
\end{corollary}

\begin{proof}
(i) Let $\Lambda\colon\Lipa(X,E^*)\to (X\widehat{\boxtimes}_\alpha E)^*$ be the isometric isomorphism defined in Theorem \ref{teo-dual}. We have 
\begin{align*}
\{f_\gamma\}\to f \text{ in } (\Lipa(X,E^*),w^*)
&\Leftrightarrow \{\Lambda(f_\gamma)\}\to \Lambda(f) \text{ in } ((X\widehat{\boxtimes}_\alpha E)^*,w^*)\\
&\Leftrightarrow \left\{\left\langle \kappa_{X\widehat{\boxtimes}_\alpha E}(u),\Lambda(f_\gamma)\right\rangle\right\}\to \left\langle \kappa_{X\widehat{\boxtimes}_\alpha E}(u),\Lambda(f)\right\rangle,\quad \forall u\in X\widehat{\boxtimes}_\alpha E\\
&\Leftrightarrow \left\{\Lambda(f_\gamma)(u)\right\}\to \Lambda(f)(u),\quad \forall u\in X\widehat{\boxtimes}_\alpha E\\
&\Leftrightarrow \left\{u(f_\gamma)\right\}\to u(f),\quad \forall u\in X\widehat{\boxtimes}_\alpha E.
\end{align*}
(ii) Let $\{f_\gamma\}$ be a net in $\Lipa(X,E^*)$ which converges in the w* topology to $f\in\Lipa(X,E^*)$. By (i), $\{u(f_\gamma)\}$ converges to $u(f)$ for each $u\in X\widehat{\boxtimes}_\alpha E$. In particular, $\{u(f_\gamma)\}$ converges to $u(f)$ for each $u\in X\boxtimes E$ since $X\boxtimes E\subset X\widehat{\boxtimes}_\alpha E$. This means that $\{f_\gamma\}$ converges to $f$ in the w*Lo topology. Therefore the identity on $\Lipa(X,E^*)$ is a continuous bijection from the w* topology to the w*Lo topology and thus the latter topology is weaker that the former as required. On a bounded subset of $\Lipa(X,E^*)$, the w* topology is compact and the w*Lo topology is Hausdorff, so both topologies must coincide.
\end{proof}






\section{Cross-norm-Lipschitz approximable operators}\label{frlo}

The concepts of Lipschitz finite-rank operators and Lipschitz approximable operators from $X$ into $E$ were introduced in \cite{jsv}. The following facts and their proofs can be found in \cite{ccjv} where these classes of Lipschitz operators were studied in the case of $E^*$-valued operators on $X$. 

Let us recall that a Lipschitz operator $f\in\Lipo(X,E^*)$ is said to be Lipschitz finite-rank if the linear span $\lin(f(X))$ of $f(X)$ in $E^*$ is finite dimensional; in that case the \emph{rank of} $f$, denoted by $\rank(f)$, is defined as the dimension of $\lin(f(X))$. We denote by $\LipoF(X,E^*)$ the linear space of all Lipschitz finite-rank operators from $X$ into $E^*$. For every $g\in X^\#$ and $\phi\in E^*$, the map $g\cdot\phi\colon X\to E^*$, defined by $(g\cdot\phi)(x)=g(x)\phi$ for all $x\in X$, is in $\LipoF(X,E^*)$ with $\Lip(g\cdot\phi)=\Lip(g)\left\|\phi\right\|$ by \cite[Lemma 1.5]{ccjv}. Furthermore, every operator $f\in\LipoF(X,E^*)$ can be expressed in the form $f=\sum_{j=1}^m g_j\cdot \phi_j$, where $m=\rank(f)$, $g_1,\ldots,g_m\in X^\#$ and $\phi_1,\ldots,\phi_m\in E^*$. 

We study the relation between Lipschitz finite-rank operators and cross-norm-Lipschitz operators of $X$ into $E^*$.

\begin{theorem}\label{teo-finite-range}
Let $\alpha$ be a dualizable Lipschitz cross-norm on $X\boxtimes E$. For every $g\in X^\#$ and $\phi\in E^*$, the map $g\cdot\phi$ belongs to $\Lipa(X,E^*)$ and $\Lipa(g\cdot\phi)=\Lip(g)||\phi||$. As a consequence, $\LipoF(X,E^*)$ is contained in $\Lipa(X,E^*)$.
\end{theorem}

\begin{proof}
Let $g\in X^\#$ and $\phi\in E^*$. Since the Lipschitz injective norm $\varepsilon$ is the least dualizable Lipschitz cross-norm on $X\boxtimes E$ by \cite[Theorem 5.2]{ccjv}, we have 
\begin{align*}
\left|\sum_{i=1}^n\left\langle (g\cdot\phi)(x_i)-(g\cdot\phi)(y_i),e_i \right\rangle\right|
&=\left|\sum_{i=1}^n\left(g(x_i)-g(y_i)\right)\left\langle\phi,e_i\right\rangle\right|\\
&\leq\Lip(g)\left\|\phi\right\|\varepsilon\left(\sum_{i=1}^n\delta_{(x_i,y_i)}\boxtimes e_i\right)\\
&\leq\Lip(g)\left\|\phi\right\|\alpha\left(\sum_{i=1}^n\delta_{(x_i,y_i)}\boxtimes e_i\right)
\end{align*}
for all $\sum_{i=1}^n\delta_{(x_i,y_i)}\boxtimes e_i\in X\boxtimes E$, and so $g\cdot\phi\in\Lipa(X,E^*)$ and $\Lipa(g\cdot\phi)\leq\Lip(g)||\phi||$. The converse inequality follows from Lemma \ref{lem-alpha-Lipschitz}. Since the Lipschitz operators $g\cdot\phi$ generate linearly the space 
$\LipoF(X,E^*)$ and $\Lipa(X,E^*)$ is a linear space, we conclude that $\LipoF(X,E^*)$ is contained in $\Lipa(X,E^*)$. 
\end{proof}

Let us recall (see \cite{jsv}) that a Lipschitz operator from $X$ into $E^*$ is said to be \emph{Lipschitz approximable} if it is the limit in the Lipschitz norm $\Lip$ of a sequence of Lipschitz finite-rank operators from $X$ to $E^*$. Since the Banach spaces $(\Lipo(X,E^*),\Lip)$ and $(\Lipp(X,E^*),\Lipp)$ coincide by Lemma \ref{lemma-Lipschitz-pi}, it is natural to introduce the following class of Lipschitz operators. 

\begin{definition}
Let $\alpha$ be a dualizable Lipschitz cross-norm on $X\boxtimes E$. A Lipschitz operator $f\in\Lipa(X,E^*)$ is said to be \emph{$\alpha$-Lipschitz approximable} if it is the limit in the $\alpha$-Lipschitz norm $\Lipa$ of a sequence of Lipschitz finite-rank operators from $X$ to $E^*$.  
\end{definition}

Therefore the space of all $\alpha$-Lipschitz approximable operators from $X$ into $E^*$, provided that $\alpha$ is a dualizable Lipschitz cross-norm on $X\boxtimes E$, is the closure of the space $\LipoF(X,E^*)$ in $\left(\Lipa(X,E^*),\Lipa\right)$. 

\begin{theorem}\label{teo-finite-range2}
Let $\alpha$ be a dualizable Lipschitz cross-norm on $X\boxtimes E$ and let $\alpha'$ be the Lipschitz norm associated of $\alpha$.
\begin{enumerate}
\item $\left(\LipoF(X,E^*),\Lipa\right)$ is isometrically isomorphic to $X^\#\boxast_{\alpha'}E^*$, via the map $K\colon X^\#\boxast_{\alpha'}E^*\to\LipoF(X,E^*)$ given by
$$
K\left(\sum_{j=1}^m g_j\boxtimes\phi_j\right)=\sum_{j=1}^m g_j\cdot\phi_j.
$$
\item The space of all $\alpha$-Lipschitz approximable operators from $X$ into $E^*$ is isometrically isomorphic to $X^\#\widehat{\boxast}_{\alpha'}E^*$.
\end{enumerate}
\end{theorem}

\begin{proof}
By \cite[Theorem 2.5]{ccjv}, the map $K\colon X^\#\boxast E^*\to\LipoF(X,E^*)$, given by
$$
K\left(\sum_{j=1}^m g_j\boxtimes\phi_j\right)=\sum_{j=1}^m g_j\cdot\phi_j,
$$
is a linear bijection. For any $\sum_{j=1}^m g_j\boxtimes\phi_j\in X^\#\boxast E^*$, we have
\begin{align*}
\alpha'\left(\sum_{j=1}^m g_j\boxtimes\phi_j\right)
&=\sup\left\{\left|\left(\sum_{j=1}^m g_j\boxtimes\phi_j\right)\left(\sum_{i=1}^n \delta_{(x_i,y_i)}\boxtimes e_i\right)\right|
\colon \alpha\left(\sum_{i=1}^n \delta_{(x_i,y_i)}\boxtimes e_i\right)\leq 1\right\}\\
&=\sup\left\{\left|\sum_{i=1}^n\left\langle \left(\sum_{j=1}^m g_j\cdot\phi_j\right)(x_i)-\left(\sum_{j=1}^m g_j\cdot\phi_j\right)(y_i),e_i \right\rangle\right|
\colon \alpha\left(\sum_{i=1}^n \delta_{(x_i,y_i)}\boxtimes e_i\right)\leq 1\right\}\\
&=\Lipa\left(\sum_{j=1}^m g_j\cdot\phi_j\right),
\end{align*}
by using \cite[Lemmas 2.2 and 1.4]{ccjv} and Remark \ref{remark-def-cross-norm-Lipschitz operator}. Hence $K$ is an isometry from $X^\#\boxast_{\alpha'}E^*$ onto $(\LipoF(X,E^*),\Lipa)$ and this proves (i). Then (ii) follows from (i) by applying a known result of functional analysis. 
\end{proof}

From Theorems \ref{teo-dual} and \ref{teo-finite-range2}, we deduce the following consequence.

\begin{corollary}
Let $\alpha$ be a dualizable Lipschitz cross-norm on $X\boxtimes E$. Then $X^\#\widehat{\boxast}_{\alpha'}E^*$ is isometrically isomorphic to $\left(X\widehat{\boxtimes}_\alpha E\right)^*$ if and only if every $\alpha$-Lipschitz operator from $X$ to $E^*$ is $\alpha$-Lipschitz approximable.
\end{corollary}


\section{Lipschitz operator Banach ideals}\label{lobi}

We now formalize the notion of an ideal of Lipschitz operators, 
with a definition inspired by the analogous one for linear operators between Banach spaces.

\begin{definition}\label{def-ideal}
A Banach ideal of Lipschitz operators (or simply a Lipschitz operator Banach ideal) from $X$ to $E^*$ is a linear subspace $A(X,E^*)$ of $\Lipo(X,E^*)$ equipped with a norm $||\cdot||_A$ with the following properties:
\begin{enumerate}
\item The Lipschitz rank-one operator $g\cdot\phi$ from $X$ to $E^*$ belongs to $A(X,E^*)$ for every $g\in X^\#$ and $\phi\in E^*$, and $||g\cdot\phi||_A\leq\Lip(g)||\phi||$.
\item $\left(A(X,E^*),||\cdot||_A\right)$ is a Banach space.
\item The ideal property: If $f\in A(X,E^*)$, $h\in\Lipo(X,X)$ and $S\in\L(E^*,E^*)$, then the composition $Sfh$ belongs to $A(X,E^*)$ and $||Sfh||_A\leq \left\|S\right\|\left\|f\right\|_A\Lip(h)$.
\end{enumerate}
\end{definition}


Our aim is to study when $\Lipa(X,E^*)$ is a Lipschitz operator Banach ideal. We will need the following lemma.

\begin{lemma}\label{lem-Lipschitz-1}
Let $\alpha$ be a Lipschitz cross-norm on $X\boxtimes E$ and let $\sum_{i=1}^n\delta_{(x_i,y_i)}\boxtimes e_i\in X\boxtimes E$. Then there exists a $f\in\Lipa(X,E^*)$ such that $\Lipa(f)=1$ and $\sum_{i=1}^n\langle f(x_i)-f(y_i),e_i\rangle=\alpha\left(\sum_{i=1}^n\delta_{(x_i,y_i)}\boxtimes e_i\right)$.
\end{lemma}

\begin{proof}
By the Hahn--Banach theorem, there exists $\varphi\in (X\widehat{\boxtimes}_{\alpha}E)^*$ with $||\varphi||=1$ such that $\varphi\left(\sum_{i=1}^n\delta_{(x_i,y_i)}\boxtimes e_i\right)=\alpha\left(\sum_{i=1}^n\delta_{(x_i,y_i)}\boxtimes e_i\right)$. By Theorem \ref{teo-dual}, there exists a function $\Lambda^{-1}(\varphi)\in\Lipa(X,E^*)$ such that $\Lipa(\Lambda^{-1}(\varphi))=||\varphi||$ and 
$\left(\sum_{i=1}^n\delta_{(x_i,y_i)}\boxtimes e_i\right)(\Lambda^{-1}(\varphi))=\varphi\left(\sum_{i=1}^n\delta_{(x_i,y_i)}\boxtimes e_i\right)$. Take $f=\Lambda^{-1}(\varphi)$ and the lemma follows. 
\end{proof}

\begin{theorem}\label{teo-caracter-ideal}
Let $\alpha$ be a Lipschitz cross-norm on $X\boxtimes E$. Then:
\begin{enumerate}
\item If $\Lipa(X,E^*)$ is a Lipschitz operator Banach ideal, then $\alpha$ is uniform.  
\item If $\alpha$ is uniform and $E$ is a reflexive Banach space, then $\Lipa(X,E^*)$ is a Lipschitz operator Banach ideal.
\end{enumerate}
\end{theorem}

\begin{proof}
(i) Assume that $\Lipa(X,E^*)$ is a Lipschitz operator Banach ideal. Fix $\sum_{i=1}^n \delta_{(x_i,y_i)}\boxtimes e_i\in X\boxtimes E$ and let $h\in\Lipo(X,X)$ and $T\in\L(E,E)$. By Lemma \ref{lem-Lipschitz-1}, there exists $f\in\Lipa(X,E^*)$ with $\Lipa(f)=1$ such that  
$$
\sum_{i=1}^n\left\langle f(h(x_i))-f(h(y_i)),T(e_i)\right\rangle
=\alpha\left(\sum_{i=1}^n\delta_{(h(x_i),h(y_i))}\boxtimes T(e_i)\right),
$$
that is
$$
\sum_{i=1}^n\left\langle T^*fh(x_i)-T^*fh(y_i),e_i\right\rangle
=\alpha\left(\sum_{i=1}^n\delta_{(h(x_i),h(y_i))}\boxtimes T(e_i)\right),
$$
where $T^*$ denotes the adjoint operator of $T$.
Since $\Lipa(X,E^*)$ has the ideal property, then $T^*fh$ belongs to $\Lipa(X,E^*)$ and $\Lipa(T^*fh)\leq ||T^*||\Lipa(f)\Lip(h)$. Then we have
\begin{align*}
\alpha\left(\sum_{i=1}^n\delta_{(h(x_i),h(y_i))}\boxtimes T(e_i)\right)
&\leq\Lipa(T^*fh)\alpha\left(\sum_{i=1}^n \delta_{(x_i,y_i)}\boxtimes e_i\right)\\
&\leq\left\|T\right\|\Lip(h)\alpha\left(\sum_{i=1}^n \delta_{(x_i,y_i)}\boxtimes e_i\right),
\end{align*}
and so $\alpha$ is uniform. 

(ii) Notice that $\Lipa(X,E^*)$ is a linear subspace of $\Lipo(X,E^*$) and $(\Lipa(X,E^*),\Lipa)$ is a normed space which satisfies the conditions (i) and (ii) of Definition \ref{def-ideal} by Lemmas \ref{lem-alpha-Lipschitz} and \ref{lem-Lipa-norm-space} and Theorems \ref{teo-dual} and \ref{teo-finite-range}. Assume that $\alpha$ is uniform and $E$ is reflexive. We only need to prove that $\Lipa(X,E^*)$ has the ideal property. Let $f\in\Lipa(X,E^*)$, $h\in\Lipo(X,X)$ and $S\in\L(E^*,E^*)$. Since $E$ is reflexive, there exists $T\in\L(E,E)$ such that $T^*=S$ and $||T||=||S||$. 
For every $\sum_{i=1}^n \delta_{(x_i,y_i)}\boxtimes e_i\in X\boxtimes E$, we have  
\begin{align*}
\left|\sum_{i=1}^n\left\langle Sfh(x_i)-Sfh(y_i),e_i\right\rangle\right|
&=\left|\sum_{i=1}^n\left\langle fh(x_i)-fh(y_i),T(e_i)\right\rangle\right|\\
&\leq\Lipa(f)\alpha\left(\sum_{i=1}^n \delta_{(h(x_i),h(y_i))}\boxtimes T(e_i)\right)\\
&\leq\Lipa(f)\Lip(h)\left\|T\right\|\alpha\left(\sum_{i=1}^n \delta_{(x_i,y_i)}\boxtimes e_i\right).
\end{align*} 
It follows that $Sfh$ is in $\Lipa(X,E^*)$ and $\Lipa(Sfh)\leq ||S||\Lipa(f)\Lip(h)$. This completes the proof.
\end{proof}

Theorem \ref{teo-caracter-ideal} shows that there is an ``almost equivalence'' between the uniformity of $\alpha$ and $\Lipa$ being an ideal.
The situation will be cleaner in Section \ref{rtmlobi}, when we consider tensor norms and ideals defined not just for a fixed metric space and a fixed Banach space, but rather for all pairs of such spaces.

We now study when $\alpha$-Lipschitz approximable operators from $X$ into $E^*$ form a Lipschitz operator Banach ideal.

\begin{theorem}\label{theo-finite-ideal}
Let $\alpha$ be a uniform and dualizable Lipschitz cross-norm on $X\boxtimes E$. Assume that $E$ is a reflexive Banach space. Then $\left(\overline{\LipoF}(X,E^*),\Lipa\right)$ is a Lipschitz operator Banach ideal. 
\end{theorem}

\begin{proof}
We first show that $\left(\LipoF(X,E^*),\Lipa\right)$ has the ideal property. Let $f\in\LipoF(X,E^*)$, $h\in\Lipo(X,X)$ and $S\in\L(E^*,E^*)$. Since $\lin(Sfh(X))=S(\lin(fh(X)))\subset S(\lin(f(X)))$, we infer that $Sfh\in\LipoF(X,E^*)$. The inequality $\Lipa(Sfh)\leq ||S||\Lipa(f)\Lip(h)$ follows similarly as in the proof of the assertion (ii) of Theorem \ref{teo-caracter-ideal}.

By Theorems \ref{teo-finite-range} and \ref{teo-dual}, $\left(\overline{\LipoF}(X,E^*),\Lipa\right)$ satisfies the conditions (i) and (ii) of Definition \ref{def-ideal}. In order to prove that it has the ideal property, let $f\in\overline{\LipoF}(X,E^*)$, $h\in\Lipo(X,X)$ and $S\in\L(E^*,E^*)$. Then we can take a sequence $\{f_n\}$ in $\LipoF(X,E^*)$ such that $\Lipa(f_n-f)\to 0$. Then $\Lipa(Sf_nh-Sfh)\to 0$ since 
$$
\Lipa(Sf_nh-Sfh)=\Lipa(S(f_n-f)h)\leq\left\|S\right\|\Lipa(f_n-f)\Lip(h)
$$ 
for all $n\in\N$. Hence $Sfh\in\overline{\LipoF}(X,E^*)$. Since $\Lipa(Sf_nh)\leq ||S||\Lipa(f_n)\Lip(h)$ for all $n\in\N$, we deduce that $\Lipa(Sfh)\leq ||S||\Lipa(f)\Lip(h)$, and the theorem follows.
\end{proof}

\section{Lipschitz operator Banach spaces}\label{lobs}

In Theorem \ref{teo-dual} we have characterized $\Lipa(X,E^*)$ as the dual space $(X\widehat{\boxtimes}_\alpha E)^*$. Our aim in this section is to tackle the general duality problem as to when a space of maps from $X$ to $E^*$ is isometrically isomorphic to $(X\widehat{\boxtimes}_\alpha E)^*$ for some Lipschitz cross-norm $\alpha$, regardless of whether or not one has an ideal property. For that purpose, we first introduce Banach spaces of Lipschitz operators. 

\begin{definition}\label{def-Lipschitz operator Banach space}
A Banach space of Lipschitz operators (or simply a Lipschitz operator Banach space) from $X$ to $E^*$ is a linear subspace $B(X,E^*)$ of $\Lipo(X,E^*)$ equipped with a norm $||\cdot||_B$ having the following properties:
\begin{enumerate}
\item $\left(B(X,E^*),||\cdot||_B\right)$ is a Banach space.
\item $||f||_B\geq\Lip(f)$ for every $f\in B(X,E^*)$.
\item For every $g\in X^\#$ and $\phi\in E^*$, the map $g\cdot\phi$ belongs to $B(X,E^*)$ and $||g\cdot\phi||_B=\Lip(g)||\phi||$.
\end{enumerate}
\end{definition}

We first characterize all Lipschitz cross-norms $\alpha$ on $X\boxtimes E$ for which $\Lipa(X,E^*)$ is a Lipschitz operator Banach space.

\begin{theorem}\label{teo-caracter-Lipschitz operator Banach space}
Let $\alpha$ be a Lipschitz cross-norm on $X\boxtimes E$. Then $\Lipa(X,E^*)$ is a Lipschitz operator Banach space if and only if $\alpha$ is dualizable.
\end{theorem}

\begin{proof}
In view of Lemmas \ref{lem-Lipa-norm-space} and \ref{lem-alpha-Lipschitz} and Theorem \ref{teo-dual}, $\Lipa(X,E^*)$ is a linear subspace of $\Lipo(X,E^*)$ and $\left(\Lipa(X,E^*),\Lipa\right)$ is a normed space satisfying assumptions (i) and (ii) of Definition \ref{def-Lipschitz operator Banach space}. Hence we only need to prove that $\Lipa(X,E^*)$ satisfies condition (iii) if and only if $\alpha$ is dualizable. 

If $\alpha$ is dualizable, then $\Lipa(X,E^*)$ has the property (iii) by Theorem \ref{teo-finite-range}. Conversely, assume that every map $g\cdot\phi\colon X\to E^*$, with $g\in X^\#$ and $\phi\in E^*$, is in $\Lipa(X,E^*)$ and $\Lipa(g\cdot\phi)=\Lip(g)||\phi||$. Take $g\in X^\#$ and $\phi\in E^*$, and since
\begin{align*}
\left|\sum_{i=1}^n (g(x_i)-g(y_i))\left\langle\phi,e_i\right\rangle\right|
&=\left|\sum_{i=1}^n\left\langle(g\cdot\phi)(x_i)-(g\cdot\phi)(y_i),e_i\right\rangle\right|\\
&\leq\Lipa(g\cdot\phi)\alpha\left(\sum_{i=1}^n \delta_{(x_i,y_i)}\boxtimes e_i\right)\\
&=\Lip(g)\left\|\phi\right\|\alpha\left(\sum_{i=1}^n \delta_{(x_i,y_i)}\boxtimes e_i\right)
\end{align*}
for all $\sum_{i=1}^n \delta_{(x_i,y_i)}\boxtimes e_i\in X\boxtimes E$, then $\alpha$ is dualizable.
\end{proof}

Since $\pi$, $\varepsilon$, $d_p$ and $w_p$ for $p\in [1,\infty]$ are dualizable Lipschitz cross-norms on $X\boxtimes E$, Theorem \ref{teo-caracter-Lipschitz operator Banach space} gives the following.

\begin{corollary}
The spaces $\Lipa(X,E^*)$ for $\alpha=\pi,\varepsilon ,d_p,w_p$ with $1\leq p\leq\infty$ are Lipschitz operator Banach spaces.
\end{corollary}


Conversely, we will now address the problem of when a Lipschitz operator Banach space can be canonically isometrically identified with the dual of a Lipschitz tensor product endowed with a Lipschitz cross-norm. We begin with the following lemma. 

\begin{lemma}\label{lem-end}
Let $B(X,E^*)$ be a Lipschitz operator Banach space. For each $u=\sum_{i=1}^n \delta_{(x_i,y_i)}\boxtimes e_i\in X\boxtimes E$, define  
$$
\alpha(u)
=\sup\left\{\left|\sum_{i=1}^n\left\langle f(x_i)-f(y_i),e_i\right\rangle\right|\colon f\in B(X,E^*),\; \left\|f\right\|_B=1\right\}
$$ 
and 
$$
\left\langle i(u),f\right\rangle
=\sum_{i=1}^n\left\langle f(x_i)-f(y_i),e_i\right\rangle
\qquad \left(f\in B(X,E^*)\right).
$$
Then $\alpha$ is a dualizable Lipschitz cross-norm on $X\boxtimes E$ and $i$ is a linear isometry from $X\boxtimes_\alpha E$ into $B(X,E^*)^*$.
\end{lemma}

\begin{proof}
Let $u=\sum_{i=1}^n \delta_{(x_i,y_i)}\boxtimes e_i\in X\boxtimes E$ and $f\in B(X,E^*)$. Note that $\langle i(u),f\rangle=u(f)$. Clearly, $i(u)$ is well defined on $B(X,E^*)$, it is linear and $||\langle i(u),f\rangle||\leq\Lip(f)\pi(u)\leq ||f||_B\pi(u)$ for all $f\in B(X,E^*)$. Then $i(u)$ is in $B(X,E^*)^*$ and 
$$
\left\|i(u)\right\|:=\sup\left\{\left|\left\langle i(u),f\right\rangle\right|\colon f\in B(X,E^*),\; \left\|f\right\|_B=1\right\}\leq\pi(u). 
$$
It is immediate that $i\colon X\boxtimes E\to B(X,E^*)^*$ is well defined and linear. Moreover, it is injective. Indeed, $i(u)=0$ means that $\langle i(u),f\rangle=0$ for all $f\in B(X,E^*)$. Since $B(X,E^*)$ contains the maps $g\cdot\phi$, it follows that $\langle u,g\cdot\phi\rangle=\langle i(u),g\cdot\phi\rangle=0$ for all $g\in X^\#$ and $\phi\in E^*$, and then $u=0$ by \cite[Proposition 1.6]{ccjv}.

Define the map $\alpha$ on $X\boxtimes E$ as in the statement. Notice that $\alpha(u)=||i(u)||$. Then $\alpha$ is a norm on $X\boxtimes E$ and so $i$ is a linear isometry from $X\boxtimes_\alpha E$ into $B(X,E^*)^*$. 

We claim that $\alpha$ is a Lipschitz cross-norm. Indeed, for any $\delta_{(x,y)}\boxtimes e\in X\boxtimes E$, we have 
$$
\alpha\left(\delta_{(x,y)}\boxtimes e\right)
=\left\|i\left(\delta_{(x,y)}\boxtimes e\right)\right\|
\leq \pi\left(\delta_{(x,y)}\boxtimes e\right)
=d(x,y)\left\|e\right\|.
$$
For the reverse, we may take $\phi\in S_{E^*}$ and $g\in S_{X^\#}$ such that $\langle\phi,e\rangle=||e||$ and $g(x)-g(y)=d(x,y)$. For example, $g(z)=d(z,y)-d(0,y)$ for all $z\in X$. Then $g\cdot\phi\in B(X,E^*)$ with $||g\cdot\phi||=1$, and we infer that   
$$
\alpha\left(\delta_{(x,y)}\boxtimes e\right)
\geq \left|\left\langle(g\cdot\phi)(x)-(g\cdot\phi)(y),e\right\rangle\right|
=\left|(g(x)-g(y))\left\langle \phi,e\right\rangle\right|
=d(x,y)\left\|e\right\|,
$$
and this proves our claim.

Finally, we prove that $\alpha$ is dualizable. Let $u=\sum_{i=1}^n \delta_{(x_i,y_i)}\boxtimes e_i\in X\boxtimes E$. For any $g\in S_{X^\#}$ and $\phi\in S_{E^*}$, we have
\begin{align*}
\left|\sum_{i=1}^n(g(x_i)-g(y_i))\left\langle\phi,e_i\right\rangle\right|
&=\left|\sum_{i=1}^n\left\langle(g\cdot\phi)(x_i)-(g\cdot\phi)(y_i),e_i\right\rangle\right|\\
&\leq\sup\left\{\left|\sum_{i=1}^n\left\langle f(x_i)-f(y_i),e_i\right\rangle\right|\colon f\in B(X,E^*),\; \left\|f\right\|_B=1\right\}
\end{align*}
and therefore $\varepsilon(u)\leq\alpha(u)$. Then $\alpha$ is dualizable by \cite[Theorem 6.3 and Proposition 6.4]{ccjv}.
\end{proof}

We are ready to obtain the main result of this section. 

\begin{theorem}\label{teo-carac-dual-Lipschitz operator Banach space}
Let $B(X,E^*)$ be a Lipschitz operator Banach space. Then the following are equivalent:
\begin{enumerate}
	\item There exists a dualizable Lipschitz cross-norm $\alpha$ on $X\boxtimes E$ such that $B(X,E^*)=\Lipa(X,E^*)$ and $||f||_B=\Lipa(f)$ for every $f\in B(X,E^*)$.
	\item If $f$ is in $\Lipo(X,E^*)$ and $\{f_\gamma\}$ is a bounded net in $B(X,E^*)$ which converges to $f$ in the relative weak* Lipschitz operator topology of $\Lipo(X,E^*)$ on $B(X,E^*)$, then $f\in B(X,E^*)$ and $||f||_B\leq\sup\{||f_\gamma||_B\colon\gamma\in\Gamma\}$.
\end{enumerate} 
\end{theorem}

\begin{proof}
Suppose that (i) holds. Let $f\in\Lipo(X,E^*)$ and let $\{f_\gamma\}$ be a bounded net in $B(X,E^*)$ converging to $f$ in the relative w*Lo topology of $\Lipo(X,E^*)$ on $B(X,E^*)$ . Denote $M=\sup\{||f_\gamma||_B\colon\gamma\in\Gamma\}$. If $\sum_{i=1}^n \delta_{(x_i,y_i)}\boxtimes e_i\in X\boxtimes E$ and $\varepsilon>0$, we have 
$$
\left|\sum_{i=1}^n\left\langle f(x_i)-f(y_i),e_i\right\rangle-\sum_{i=1}^n\left\langle f_{\gamma_0}(x_i)-f_{\gamma_0}(y_i),e_i\right\rangle\right|
=\left|\left(\sum_{i=1}^n \delta_{(x_i,y_i)}\boxtimes e_i\right)(f-f_{\gamma_0})\right|
<\varepsilon
$$
for some $\gamma_0\in\Gamma$, and therefore
\begin{align*}
\left|\sum_{i=1}^n\left\langle f(x_i)-f(y_i),e_i\right\rangle\right|
&<\left|\sum_{i=1}^n\left\langle f_{\gamma_0}(x_i)-f_{\gamma_0}(y_i),e_i\right\rangle\right|+\varepsilon\\
&\leq\Lipa(f_{\gamma_0})\alpha\left(\sum_{i=1}^n\delta_{(x_i,y_i)}\boxtimes e_i\right)+\varepsilon\\
&=\left\|f_{\gamma_0}\right\|_B\alpha\left(\sum_{i=1}^n\delta_{(x_i,y_i)}\boxtimes e_i\right)+\varepsilon\\
&\leq M\alpha\left(\sum_{i=1}^n\delta_{(x_i,y_i)}\boxtimes e_i\right)+\varepsilon.
\end{align*}
Since $\varepsilon$ was arbitrary, we deduce that $f\in\Lipa(X,E^*)$ and $\Lipa(f)\leq M$. Hence $f\in B(X,E^*)$ and $||f||_B\leq M$.

Conversely, assume that (ii) is true. Take the dualizable Lipschitz cross-norm $\alpha$ on $X\boxtimes E$ and the linear isometry $i$ from $X\boxtimes_\alpha E$ into $B(X,E^*)^*$ defined in Lemma \ref{lem-end}. Next we check that $B(X,E^*)=\Lipa(X,E^*)$ and $||f||_B=\Lipa(f)$ for all $f\in B(X,E^*)$. To this end, we first take a function $f$ in $B(X,E^*)$. The definition of $\alpha$ gives  
$$
\left|\sum_{i=1}^n\left\langle f(x_i)-f(y_i),e_i\right\rangle\right|
\leq\left\|f\right\|_B\alpha\left(\sum_{i=1}^n\delta_{(x_i,y_i)}\boxtimes e_i\right)
$$
for all $\sum_{i=1}^n\delta_{(x_i,y_i)}\boxtimes e_i\in X\boxtimes E$, and then $f\in\Lipa(X,E^*)$ with $\Lipa(f)\leq ||f||_B$. Conversely, pick a function $f$ in $\Lipa(X,E^*)$ and consider the functional $S(f)\colon i(X\boxtimes E)\to\K$ given by 
$$
\left\langle S(f),i(u)\right\rangle=\sum_{i=1}^n\left\langle f(x_i)-f(y_i),e_i\right\rangle
$$
for $u=\sum_{i=1}^n\delta_{(x_i,y_i)}\boxtimes e_i\in X\boxtimes E$. The fact that $i$ is injective guarantees that $S(f)$ is well defined. The linearity of $S(f)$ follows easily. Since $|\langle S(f),i(u)\rangle|
=|u(f)|
\leq\Lipa(f)\alpha(u)
=\Lipa(f)||i(u)||$ for all $u\in X\boxtimes E$, it follows that $S(f)$ is continuous and $||S(f)||\leq\Lipa(f)$. Since $i(X\boxtimes E)$ is a linear subspace of $ B(X,E^*)^*$, the Hahn--Banach theorem provides a functional $\widetilde{S}(f)\in B(X,E^*)^{**}$ which extends to $S(f)$ and has the same norm.
Let $\kappa_B$ be the canonical injection from $B(X,E^*)$ into $B(X,E^*)^{**}$. By Goldstein's theorem, there exists a net $\{f_\gamma\}$ in $B(X,E^*)$ for which $\sup\{||f_\gamma||_B\colon\gamma\in\Gamma\}\leq ||\widetilde{S}(f)||$ and $\{\kappa_B(f_\gamma)\}$ converges to $\widetilde{S}(f)$ in the weak* topology of $B(X,E^*)^{**}$. 
Since $i(X\boxtimes E)\subset B(X,E^*)^*$, it follows that for each $u\in X\boxtimes E$, the net $\{\langle \kappa_B(f_\gamma),i(u)\rangle\}$ converges to $\langle\widetilde{S}(f),i(u)\rangle$, that is, 
$\{u(f_\gamma)\}$ converges to $u(f)$. This means that $\{f_\gamma\}$ converges to $f$ in the relative weak* Lipschitz operator topology of $\Lipo(X,E^*)$ on $B(X,E^*)$ by Remark \ref{remark new}. Then, by hypothesis, $f\in B(X,E^*)$ and $||f||_B\leq\sup\{||f_\gamma||_B\colon\gamma\in\Gamma\}\leq\Lipa(f)$. This finishes the proof.
\end{proof}

Theorem \ref{teo-carac-dual-Lipschitz operator Banach space} can be reformulated as follows.

\begin{corollary}\label{cor-carac-dual-lobs}
Let $B(X,E^*)$ be a Lipschitz operator Banach space. The following are equivalent:
\begin{enumerate}
\item There exists a dualizable Lipschitz cross-norm $\alpha$ on $X\boxtimes E$ such that $B(X,E^*)=\Lipa(X,E^*)$ and $||f||_B=\Lipa(f)$ for all $f\in B(X,E^*)$.
\item There exists a dualizable Lipschitz cross-norm $\alpha$ on $X\boxtimes E$ such that $B(X,E^*)$ is isometrically isomorphic to $(X\widehat{\boxtimes}_\alpha E)^*$.
\item The closed unit ball of $B(X,E^*)$ is compact in the weak* Lipschitz operator topology of $\Lipo(X,E^*)$.
\end{enumerate}
\end{corollary}

\begin{proof}
(i) implies (ii) is deduced immediately taking into account Theorem \ref{teo-dual}. Assume now that (ii) holds. Then $B(X,E^*)$ is isometrically isomorphic to $\Lipa(X,E^*)$ by Theorem \ref{teo-dual}. Then the Alaoglu's theorem says us that the closed unit ball of $B(X,E^*)$ is compact in the w* topology of $\Lipa(X,E^*)$ and hence, by Corollary \ref{cor-theo-dual}, in the w*Lo topology of $\Lipa(X,E^*)$. Since $\Lipa(X,E^*)$ is a linear subspace of $\Lipo(X,E^*)$, this last topology agrees with the relative w*Lo topology of $\Lipo(X,E^*)$ on $\Lipa(X,E^*)$ by Remark \ref{remark new}. Then (iii) follows easily. 

Finally, suppose that (iii) is true. Let $f\in\Lipo(X,E^*)$ and let $\{f_\gamma\}$ be a bounded net in $B(X,E^*)$ which converges to $f$ in the relative w*Lo topology of $\Lipo(X,E^*)$ on $B(X,E^*)$. Let $M=\sup\{||f_\gamma||_B\colon\gamma\in\Gamma\}$. By (iii), the net $\{f_\gamma/M\}$ has a cluster point $g$ in the closed unit ball of $B(X,E^*)$ for the w*Lo topology of $\Lipo(X,E^*)$. We now apply Remark \ref{remark new} to obtain that $g$ is a cluster point of $\{f_\gamma/M\}$ for the relative w*Lo topology of $\Lipo(X,E^*)$ on $B(X,E^*)$, 
but since $\{f_\gamma\}$ converges to $f$ in this same topology which is Hausdorff, we infer that $f/M=g$. 
Hence $f\in B(X,E^*)$ and $||f||_B\leq M$. Then the assertion (ii) of Theorem \ref{teo-carac-dual-Lipschitz operator Banach space} is satisfied and we obtain (i).
\end{proof}


\section{The representation theorem for maximal Lipschitz operator Banach ideals}\label{rtmlobi}

The aim of this section is to study the relationship between Lipschitz cross-norms and Lipschitz operator Banach ideals even in the case when the latter takes values in a Banach space which is not a dual space.
In order to do that, we prove a Lipschitz version of the representation theorem for maximal operator ideals \cite[17.5]{Defant-Floret} (originally due to H. P. Lotz \cite{Lotz}).
In a nutshell, that result says that under appropriate conditions on the operator ideal (or the tensor norm),
having a duality between an operator ideal and a tensor norm at the finite-dimensional level extends to a general duality. 
In order to even make sense of that in our Lipschitz context, we will need to extend the definitions of Lipschitz cross-norms and Lipschitz operator Banach ideals to cover all spaces at once and not just a fixed pair.

\subsection{Basic definitions and notations}

For a pointed metric space $X$ (a Banach space $E$), we denote by $\mfin(X)$ (respectively, $\fin(E)$) the set of all finite subsets of $X$ that contain the base point (respectively, the set of all finite-dimensional subspaces of $E$).
For a Banach space $E$, we denote by $\cofin(E)$ the set of all finite-codimensional subspaces of $E$.
Given $L \in \cofin(E)$, let $Q^E_L \colon E \to E/L$ be the canonical projection,
and given $Y \subset X$ let $I^X_Y \colon  Y \to X$ be the canonical injection. For a Banach space $E$, $\kappa_E \colon  E \to E^{**}$ denotes the canonical injection.

\begin{definition}\label{defn-generic-cross-norm}
By a generic Lipschitz cross-norm $\alpha$, we mean an assignment for each pointed metric space $X$ and each Banach space $E$ of a Lipschitz cross-norm $\alpha(\cdot; X, E)$ on the Lipschitz tensor product $X\boxtimes E$ 
(sometimes denoted simply by $\alpha$ if the spaces are clear from the context)
such that:
\begin{enumerate}
 \item $\alpha$ is dualizable, that is, $\varepsilon \le \alpha \le \pi$.
 \item $\alpha$ satisfies the metric mapping property:
if $h \in \Lipo(X_0,X_1)$ and $T \in \L(E_0,E_1)$, then
$$
\left\| h \boxtimes T \colon  X_0\boxtimes_\alpha E_0 \to X_1\boxtimes_\alpha E_1 \right\| \le \Lip(h)\left\|T\right\|.
$$
\end{enumerate}
If the assignment and the conditions above are given only for finite pointed metric spaces and finite-dimensional Banach spaces, we say that $\alpha$ is $\fin$-generic. 

A generic Lipschitz cross-norm $\alpha$ is said to be finitely generated if 
$$
\alpha(u ; X, E) = \inf\left\{\alpha(u; X_0,E_0) \colon X_0 \in \mfin(X), E_0 \in\fin(E), u \in X_0 \boxtimes E_0\right\}
$$
for every pointed metric space $X$, every Banach space $E$ and every $u \in X \boxtimes E$.
\end{definition}

Note that condition (ii) in Definition \ref{defn-generic-cross-norm} is a generalization of uniformity,
and that all the Lipschitz cross-norms we have defined until now,  namely $\varepsilon$, $\pi$, $d_p$ and $w_p$, are in fact finitely generated generic Lipschitz cross-norms.

Given a $\fin$-generic Lipschitz cross-norm $\alpha$, we can use the following procedure to extend it to a finitely generated generic Lipschitz cross-norm.

\begin{lemma}\label{lemma-finitely-generated-cross-norm}
Let $\alpha$ be a $\fin$-generic Lipschitz cross-norm.
For a pointed metric space $X$, a Banach space $E$ and $u \in X \boxtimes E$,
define
$$
\theta(u; X, E) = \inf\left\{\alpha(u ; X_0, E_0) \colon X_0 \in \mfin(X), E_0 \in \fin(E), u \in X_0 \boxtimes E_0\right\}.
$$
Then $\theta$ is a finitely generated generic Lipschitz cross-norm.
Moreover, $\theta$ and $\alpha$ coincide on $X \boxtimes E$ whenever $X$ is a finite pointed metric space and $E$ is a finite-dimensional Banach space.
\end{lemma}

\begin{proof}
First, let us show that $\theta$ is a norm on $X \boxtimes E$.
It is clear that $\theta$ satisfies $\theta(\lambda u) = |\lambda| \theta(u)$, because so does $\alpha$.
Let $u_1, u_2 \in X \boxtimes E$.
Take $X_1, X_2 \in \mfin(X)$, $E_1, E_2 \in \fin(E)$ such that $u_j \in X_j \boxtimes E_j$, $j=1,2$.
Then, if $X_0 = X_1 \cup X_2$ and $E_0 = E_1 + E_2$, using the metric mapping property of $\alpha$ applied to the inclusions
$X_j \to X_0$ and $E_j \to E_0$, $j=1,2$,
we have
$$
\alpha(u_1+u_2 ; X_0, E_0) 
\le \alpha(u_1 ; X_0, E_0) + \alpha(u_2 ; X_0, E_0)
\le \alpha(u_1 ; X_1, E_1) + \alpha(u_2 ; X_2, E_2)
$$
and, by taking the infimum over all such $X_j$, $E_j$, we conclude that
$$
\theta(u_1+u_2 ; X, E) \le \theta(u_1 ; X, E) + \theta(u_2 ; X, E),
$$
giving the triangle inequality.
Since $\varepsilon$ and $\pi$ are finitely generated and $\varepsilon \le \alpha \le \pi$, it follows that $\varepsilon \le \theta \le \pi$
and thus $\theta$ is a dualizable Lipschitz cross-norm on $X \boxtimes E$.

Now let $X,Y$ be pointed metric spaces, $E,F$ be Banach spaces, $h \in \Lipo(X,Y)$ and $S \in \L(E,F)$.
Let $u \in X \boxtimes E$.
Given $X_0 \in \mfin(X)$ and $E_0\in \fin(E)$ such that $u \in X_0 \boxtimes E_0$,
note that $Y_0:= h(X_0) \in \mfin(Y)$, $F_0 := S(E_0) \in \fin(F)$ and $(h\boxtimes S)(u) \in Y_0 \boxtimes F_0$.
From the metric mapping property for $\alpha$, we infer that
$$
\theta\big( (h\boxtimes S)(u) ; Y,F \big)
\le \alpha\big( (h\boxtimes S)(u) ; Y_0,F_0 \big)
\le \Lip(h)\left\|S\right\|\alpha\big( u ; X_0,E_0 \big),
$$
and, by taking the infimum over all such $X_0, E_0$, we conclude that $\theta$ has the metric mapping property.

Now assume that $X$ is a finite pointed metric space, $E$ is a finite-dimensional Banach space and $u \in X \boxtimes E$.
By the definition of $\theta$, clearly $\theta(u ; X, E ) \le \alpha(u ; X , E)$.
Whenever  $X_0 \in \mfin(X)$, $E_0 \in \fin(E)$ are such that $u \in X_0 \boxtimes E_0$, the metric mapping property of $\alpha$ applied to the inclusion maps $X_0 \to X$ and $E_0 \to E$ shows that $\alpha(u ; X , E) \le \alpha(u ; X_0 , E_0)$, so we conclude that $\theta(u ; X, E ) = \alpha(u ; X , E)$.
\end{proof}

\begin{definition}\label{defn-generic-operator ideal}
By a generic Lipschitz operator Banach ideal $A$, we mean an assignment for each pointed metric space $X$ and each Banach space $E$ of a 
linear subspace $A(X, E)$ of $\Lipo(X, E)$ equipped with a norm $\|\cdot\|_A$ with the following properties:
\begin{enumerate}
 \item The Lipschitz rank-one operator $g \cdot e$ from $X$ to $E$ belongs to $A(X, E)$ for every $g \in X^\#$ and $e\in E$, and $\left\| g \cdot e\right\|_A \le \Lip(g) \left\|e\right\|$.
 \item $\Lip\le \left\|\cdot\right\|_A$.
 \item $(A(X, E), \left\|\cdot\right\|_A)$ is a Banach space.
 \item The (strengthened) ideal property: If $f \in A(X, E)$, $h \in \Lipo(Z, X)$ and $S \in \L(E, F)$, then the composition $S f h$ belongs
to $A(Z, F)$ and $\left\|Sfh\right\|_A \le \left\|S\right\|\left\|f\right\|_A\Lip(h)$.
\end{enumerate}
If the assignment and the conditions above are given only for finite pointed metric spaces and finite-dimensional Banach spaces, we say that $A$ is $\fin$-generic.
\end{definition}

Note that Definition \ref{defn-generic-operator ideal} is a combination of the definitions of Lipschitz operator Banach ideal (Definition \ref{def-ideal}) and Lipschitz operator Banach space (Definition \ref{def-Lipschitz operator Banach space}), but with a stronger ideal property and for a general Banach space instead of a dual one.
Note also that $\Lipo$, $\Pi_p^L$ and $\Gamma_p^{\Lip}$ are examples of generic Lipschitz operator Banach ideals.

Given a $\fin$-generic Lipschitz operator Banach ideal $A$, there are several different ways of extending it to a generic Lipschitz operator Banach ideal. Here we consider the ``largest'' such extension.

\begin{lemma}\label{lemma-maximal-ideal}
Let $A$ be a $\fin$-generic Lipschitz operator Banach ideal. For a pointed metric space $X$, a Banach space $E$ and $f\in\Lipo(X,E)$, define 
$$
\left\|f\right\|_{A^{\max}}= \sup\left\{\left\| Q^E_L \circ f \circ I^X_Y \right\|_A \colon Y \in \mfin(X), L \in \cofin(E)\right\}
$$
and
$$
A^{\max}(X,E)= \left\{f \in \Lipo(X,E) \colon \left\|f\right\|_{A^{\max}} < \infty\right\}.
$$
Then $(A^{\max}, \left\|\cdot\right\|_{A^{\max}})$ is a generic Lipschitz operator Banach ideal.
Moreover, $A^{\max}(X,E) = A(X,E)$ holds isometrically whenever $X$ is a finite pointed metric space and $E$ is a finite-dimensional Banach space.
\end{lemma}

\begin{proof}
Clearly, $A^{\max}(X,E)$ is a nonempty subset of $\Lipo(X,E)$. Since $\left\|\cdot\right\|_A$ is a norm, it is immediate that $A^{\max}(X,E)$ is a linear subspace of $\Lipo(X,E)$ and that $ \left\|\cdot\right\|_{A^{\max}}$ is a norm. We now verify the conditions in the definition of a generic Lipschitz operator Banach ideal:
\begin{enumerate}
 \item Let $g \in X^\#$ and $e \in E$. For every $Y \in \mfin(X)$ and $L \in \cofin(E)$,
 we have
\begin{align*}
\left\|Q^E_L \circ (g \cdot e) \circ I^X_Y \right\|_{A} &= \left\|(g|_Y) \cdot (Q^E_Le) \right\|_A \\
&\le \Lip(g|_Y) \left\| Q^E_Le \right\| \\
&\le \Lip(g)\left\| e \right\|,
\end{align*}
so $g\cdot e$ belongs to $A^{\max}(X,E)$ and $\left\|g\cdot e\right\|_{A^{\max}} \le \Lip(g)\left\|  e \right\|$.
\item Note that, for any $f \in \Lipo(X, E)$,
$$
\Lip(f) = \sup\big\{ \Lip( Q^E_L \circ f \circ I^X_Y ) \colon Y \in \mfin(X), L \in \cofin(E) \big\},
$$
from where it follows that $\Lip\le \left\|\cdot\right\|_{A^{\max}}$.
 \item Since we already know that $A^{\max}(X,E)$ is a normed space, it suffices to show that every absolutely convergent series $\sum f_n $ in $A^{\max}(X,E)$ is convergent.
Since $\Lip\le \left\|\cdot\right\|_{A^{\max}}$, the series $\sum_n f_n$ converges in $\Lipo(X,E)$ to a limit $f \in \Lipo(X,E)$.
Fix $Y \in\fin(X)$ and $L\in\cofin(E)$.
Since $Y$ is finite and $E/L$ is finite-dimensional, by Corollary \ref{cor-carac-dual-lobs} there exists a dualizable Lipschitz cross-norm $\alpha$ on $Y \boxtimes (E/L)^*$ such that $A(Y, E/L) = \big( Y \boxtimes_\alpha (E/L)^* \big)^*$.
Note that $\sum_n Q^E_L \circ f_n \circ I^X_Y$ converges pointwise to $Q^E_L \circ f \circ I^X_Y$, so
for each $u \in Y \boxtimes (E/L)^*$, we have
\begin{align*}
\big|(Q^E_L \circ f \circ I^X_Y)(u)\big|
&= \Big| \sum_{n=1}^{+\infty} (Q^E_L \circ f_n \circ I^X_Y)(u) \Big| \\
&\le  \sum_{n=1}^{+\infty} \big|(Q^E_L \circ f_n \circ I^X_Y)(u) \big| \\
&\le \alpha(u) \sum_{n=1}^{+\infty} \left\|Q^E_L \circ f_n \circ I^X_Y \right\|_A,
\end{align*}
so it follows that $\left\| Q^E_L \circ f \circ I^X_Y \right\|_A \le \sum_{n=1}^{+\infty} \left\|f_n\right\|_{A^{\max}}$ and thus
$f \in A^{\max}(X,E)$ and $\left\|f\right\|_{A^{\max}} \le \sum_{n=1}^{+\infty} \left\|f_n\right\|_{A^{\max}}$.
Now, applying the same argument to $f - \sum_{n=1}^{N} f_n$ yields 
$$
\left\| f - \sum_{n=1}^{N} f_n \right\|_{A^{\max}}
=\left\|\sum_{n=N+1}^{+\infty} f_n\right\|_{A^{\max}}
\le \sum_{n=N+1}^{+\infty} \left\|f_n\right\|_{A^{\max}},
$$
so the the series $\sum_n f_n$ converges to $f$ in $A^{\max}(X,E)$.
 \item Let $f \in A^{\max}(X, E)$, $h \in \Lipo(Z, X)$ and $S \in \L(E, F)$.
 Fix $Y \in \mfin(Z)$ and $L \in \cofin(F)$. Let $K \subset E$ be the kernel of the map $Q^F_L\circ S$, and notice that $K \in \cofin(E)$ and that, by the universal property of quotients, there is a linear map $\tilde{S} \colon  E/K \to F/L$ with $\left\|\tilde{S}\right\|\le\left\| S\right\|$ such that $ \tilde{S}Q^E_K = Q^F_LS$. Thus, noticing that $h(Y) \in \fin(X)$ and using the ideal property of $A$,
 \begin{align*}
\left\|Q^F_L \circ Sfh \circ I^Z_Y\right\|_A 
&= \left\|\tilde{S}Q^E_K \circ f \circ I^X_{h(Y)} \circ hI^Z_Y\right\|_A \\
&\le \left\|\tilde{S}\right\| \left\| Q^E_K \circ f \circ I^X_{h(Y)} \right\|_A \Lip(hI^Z_Y)\\
&\le \left\|S\right\| \left\|f \right\|_{A^{\max}} \Lip(h).
\end{align*}
\end{enumerate}
Now let $X$ be a finite pointed metric space, $E$ be a finite-dimensional Banach space and $f \in \Lipo(X,E)$.
From the definition of $A^{\max}$ and the ideal property for $A$, it is clear that $\left\|f\right\|_{A^{\max}} \le \left\|f\right\|_A$. But taking $Y=X$ and $L=\{0\}$ in the definition of $\left\|f\right\|_{A^{\max}}$ shows that
$$
\left\|f\right\|_{A^{\max}}
\ge\left\| Q^E_{\{0\}} \circ f \circ I^X_X \right\|_A
=\left\|Q^E_{\{0\}} \circ f \right\|_A
=\left\|f\right\|_A,
$$
where the last equality follows from the fact that $Q^E_{\{0\}}$ is a bijective isometry.
\end{proof}

We call $(A^{\max}, \left\|\cdot\right\|_{A^{\max}})$ the \emph{maximal hull of $A$}, and we say that a generic Lipschitz operator Banach ideal $A$ is \emph{maximal} if $(A, \left\|\cdot\right\|_A) = (A^{\max}, \left\|\cdot\right\|_{A^{\max}})$.
Note that $A^{\max}$ is always a maximal generic Lipschitz operator Banach ideal.
A first example of a maximal generic Lipschitz operator Banach ideal is given by the ideal $\Lipo$ of Lipschitz operators.
Suppose that $f \in \Lipo^{\max}(X,E)$ with norm at most $C$,
and Let $x,y$ be distinct points in $X$. Note that
$\n{Q^E_L \circ f (x) - Q^E_L \circ f(y)} \le C d(x,y)$ for every $L \in\codim(E)$.
By taking $L$ to be the kernel of a norm one functional in $E^*$ which norms $f(x)-f(y)\in E$, we conclude that $\n{f(x) - f(y)} \le C d(x,y)$ and thus $f$ is Lipschitz with norm at most $C$ as required. 
A similar but slightly more involved argument shows that $\Pi_p^L$ and $\Gamma_p^{\Lip}$ are also maximal generic Lipschitz operator Banach ideals, based on the fact that given finitely many vectors in $E$ one can find $L \in \cofin(E)$ such that the quotient $E \to E/L$ preserves the norms of those vectors.
There are also generic Lipschitz operator Banach ideals that are not maximal, for example the Lipschitz compact operators from \cite{jsv} and the Lipschitz $p$-nuclear operators from \cite{cz} (the reader is referred to those papers for the definitions).
Any Lipschitz operator belongs to the maximal hull of the ideal of Lipschitz compact operators, since every Lipschitz operator with finite domain is Lipschitz compact, but it is easy to find Lipschitz operators which are not Lipschitz compact and thus the ideal of Lipschitz compact operators is not maximal.
Similarly, using \cite[Thm. 2.1]{cz}, the existence of a linear operator from a separable Banach space to a dual Banach space which is $p$-integral but not $p$-nuclear shows that the ideal of Lipschitz $p$-nuclear operators is not maximal.

\subsection{The association between finitely generated generic Lipschitz cross-norms and maximal generic Lipschitz operator Banach ideals}

The main idea we will exploit is that to every finitely generated generic Lipschitz cross-norm one can canonically associate a maximal generic Lipschitz operator Banach ideal, and vice versa.

We say that a $\fin$-generic Lipschitz cross-norm $\alpha$ and a $\fin$-generic Lipschitz operator Banach ideal $A$ are \emph{associated}, and we write $A \sim \alpha$, if for every finite pointed metric space $X$ and every finite-dimensional Banach space $E$, the relation $A(X,E^*) = (X \boxtimes_\alpha E)^*$, or equivalently $A(X,E^*)= \Lipa(X,E^*)$, holds isometrically.

The key will be the following generalization of Theorem \ref{teo-caracter-ideal},
whose heart is the fact that the metric mapping property of $\alpha$ and the (strengthened) ideal property of $\Lipa$ are equivalent as long as we restrict ourselves to finite metric spaces and finite-dimensional Banach spaces.

\begin{proposition}\label{prop-metric-mapping-property-equivalent-to-ideal-property}
Suppose that for every finite metric space $X$ and every finite-dimensional Banach space $E$,
$\alpha$ is a norm on $X \boxtimes E$ and
$A(X,E)$ is a linear subspace of $\Lipo(X,E)$ equipped with a norm $\left\| \cdot \right\|_A$
so that
$
A(X,E) = (X \boxtimes_\alpha E^*)^*
$
holds isometrically.
Then $\alpha$ is a $\fin$-generic Lipschitz cross-norm if and only if
$A$ is a $\fin$-generic Lipschitz operator Banach ideal.
\end{proposition}

\begin{proof}
Suppose that  $A$ is a $\fin$-generic Lipschitz operator Banach ideal.
Let $X$ be a finite pointed metric space and $E$ be a finite-dimensional Banach space.
By hypothesis, $\alpha$ is already a norm on $X \boxtimes E$.
The condition $\Lip\le \left\|\cdot\right\|_A$ implies that $\alpha \le \pi$ on $X \boxtimes E$,
whereas the fact that, for every $g \in X^\#$ and $e\in E$, we have that $\left\| g \cdot e \right\|_A \le \Lip(g) \left\|e\right\|$ implies that $\varepsilon \le \alpha$ on $X \boxtimes E$.
Thus, $\alpha$ is a dualizable Lipschitz cross-norm.
A small modification of the arguments in the proof of Theorem \ref{teo-caracter-ideal}.(i) shows that $\alpha$ has the metric mapping property.

Now suppose that $\alpha$ is a $\fin$-generic Lipschitz cross-norm.
By hypothesis, $\left\|\cdot \right\|_A$ is a complete norm on $A(X, E)$.
Reversing the arguments above, the condition $\alpha \le \pi$ implies that $\Lip\le \left\|\cdot\right\|_A$ on $A(X, E)$,
whereas the condition $\varepsilon \le \alpha$ implies that, for every
$g \in X^\#$ and $e \in E$, we have that $\left\| g \cdot e \right\|_A \le \Lip(g) \left\|e\right\|$.
Finally, a small modification of the argument in the proof of Theorem \ref{teo-caracter-ideal}.(i) shows that $A$ has the (strengthened) ideal property.
\end{proof}

More generally, we have the following two lemmas that give constructions allowing us to go back and forth between generic Lipschitz cross-norms and generic Lipschitz operator Banach ideals.

\begin{lemma}\label{lemma-from-cross-norm-to-ideal}
Let $\alpha$ be a
$\fin$-generic Lipschitz cross-norm. For a pointed metric space $X$ and a Banach space $E$,
given $f \in \Lipo(X,E)$
define
$$
\left\|f\right\|_A = \sup\left\{ \Lipa( Q^E_L \circ f \circ I^X_Y )\colon Y \in \mfin(X), L \in \cofin(E) \right\}
$$
and 
$$
A(X,E) =\left\{ f \in \Lipo(X,E) \colon \left\|f\right\|_{A} < \infty \right\}.
$$
Then $A$ is a maximal generic Lipschitz operator Banach ideal associated to $\alpha$.
\end{lemma}

\begin{proof}
First, a word about the definition:
note that since $E/L$ is finite-dimensional, it is a dual space and thus it makes sense to consider the $\Lipa$ norm of the map $Q^E_L \circ f \circ I^X_Y \colon  Y \to E/L$.
Since $\alpha$ is a
$\fin$-generic Lipschitz cross-norm, Proposition \ref{prop-metric-mapping-property-equivalent-to-ideal-property} implies that $\Lipa$ is a $\fin$-generic Lipschitz operator Banach ideal.
Therefore, from Lemma \ref{lemma-maximal-ideal}, $A$ is a maximal generic Lipschitz operator Banach ideal that agrees isometrically with $\Lipa$ whenever the pointed metric space is finite and the Banach space is finite-dimensional,
so $A \sim \alpha$.
\end{proof}

\begin{lemma}\label{from-ideal-to-cross-norm}
Let $A$ be a $\fin$-generic Lipschitz operator Banach ideal.
For a pointed metric space $X$, a Banach space $E$ and $u \in X \boxtimes E$,
define
$$
\alpha(u; X, E) = \inf \left\{ \sup\left\{\left|f(u)\right| \colon \left\| f\colon X_0 \to E_0^* \right\|_A \le 1 \right\} \colon X_0 \in \mfin(X), E_0 \in \fin(E), u \in X_0 \boxtimes E_0 \right\}.
$$
Then $\alpha$ is a finitely generated generic Lipschitz cross-norm associated to $A$.
\end{lemma}

\begin{proof}
For every finite pointed metric space $X$ and every finite-dimensional Banach space $E$, consider the norm $\alpha_0( \cdot ; X, E)$ on $X \boxtimes E$ given by duality with $A$
$$
\alpha_0(u; X, E) =\left\{ \sup\left|f(u)\right| \colon \left\|f \colon  X \to E^* \right\|_A \le 1 \right\},
$$
so that 
$(X \boxtimes_{\alpha_0} E^*)^* = A(X, E)$.
From Proposition \ref{prop-metric-mapping-property-equivalent-to-ideal-property}, it follows that $\alpha_0$ is a $\fin$-generic Lipschitz cross-norm.
By definition, $\alpha$ is the obtained from $\alpha_0$ by means of the procedure in Lemma \ref{lemma-finitely-generated-cross-norm}, which implies that $\alpha$ is a finitely generated generic Lipschitz cross-norm that agrees with $\alpha_0$ on $X\boxtimes E$ whenever the pointed metric space $X$ is finite and the Banach space $E$ is finite-dimensional, so $A \sim \alpha$.
\end{proof}

The previous two lemmas show that:
\begin{enumerate}
 \item For every $\fin$-generic Lipschitz cross-norm, there is a maximal generic Lipschitz operator Banach ideal $A$ such that $A \sim \alpha.$
 \item For every $\fin$-generic Lipschitz operator Banach ideal, there is a finitely generated generic Lipschitz cross-norm $\alpha$ such that $A \sim \alpha$.
\end{enumerate}
Since both finitely generated generic Lipschitz cross-norms and maximal generic Lipschitz operator Banach ideals are determined by their behavior on finite pointed metric spaces and finite-dimensional Banach spaces, these constructions show that the relation $\sim$ is a one-to-one correspondence between finitely generated generic Lipschitz cross-norms and maximal generic Lipschitz operator Banach ideals.

\subsection{Two basic lemmas for finitely generated generic Lipschitz cross-norms}

According to \cite[Section 13]{Defant-Floret}, there are five lemmas that are basic for the understanding and use of tensor norms.
Here we prove Lipschitz versions of the two we will need later.

Every $\varphi \in (X\widehat{\boxtimes}_\pi E)^* = \Lipo(X,E^*)$ has a canonical extension $\varphi^\wedge \in (X \widehat{\boxtimes}_\pi E^{**})^* = \Lipo(X,E^{***})$, characterized by satisfying the relation
$$
\left\langle \varphi^\wedge,\delta_{(x,y)} \boxtimes v^{**}\right\rangle
=\left\langle v^{**},L_\varphi(x) - L_\varphi(y)\right\rangle ,
$$
where $L_\varphi:=\Lambda^{-1}(\varphi) \in \Lipo(X,E^*)$ is the Lipschitz operator associated to $\varphi$ given in Theorem \ref{teo-dual}. The following lemma tells us what happens if in fact $\varphi \in  (X \widehat{\boxtimes}_\alpha E)^*$,
compare to \cite[Lemma 13.2]{Defant-Floret}.

\begin{lemma}[Extension lemma]
 Let  $\varphi \in (X\widehat{\boxtimes}_\pi E)^*$ and let $\alpha$ be a finitely generated generic Lipschitz cross-norm. Then $ \varphi \in  (X\widehat{\boxtimes}_\alpha E)^*$ if and only if $\varphi^\wedge \in (X \widehat{\boxtimes}_\alpha E^{**})^*$. In this case,
 $\left\|\varphi\right\|_{(X \widehat{\boxtimes}_\alpha E)^*} = \left\| \varphi^\wedge \right\|_{(X \widehat{\boxtimes}_\alpha E^{**})^*}$.
\end{lemma}

\begin{proof}
The metric mapping property implies that the canonical inclusion map $id_X \boxtimes \kappa_E \colon  X \boxtimes_\alpha E \to X \boxtimes_\alpha E^{**}$ is contractive, and hence $\left\|\varphi\right\|_{(X \widehat{\boxtimes}_\alpha E)^*} \le \left\| \varphi^\wedge \right\|_{(X\widehat{\boxtimes}_\alpha E^{**})^*}$.

For the converse, take $u_0 \in X \boxtimes E^{**}$ and $X_0 \in \mfin(X), E_0 \in \fin(E^{**})$ such that $u_0 \in X_0 \boxtimes E_0$.
By the principle of local reflexivity (even in a weak form as in \cite[Subsection 6.5]{Defant-Floret}), for every $\varepsilon>0$ there exists $R \in \L(E_0, E)$ with $\left\|R\right\| \le 1 +\varepsilon$ such that, for all $v^{**} \in E_0$ and $x,y \in X_0$,
$$
\left\langle v^{**},L_\varphi(x) - L_\varphi(y)\right\rangle
=\left\langle L_\varphi(x) - L_\varphi(y),Rv^{**}\right\rangle .
$$
This means that
$$
\left\langle \varphi^\wedge,\delta_{(x,y)} \boxtimes v^{**}\right\rangle
=\left\langle \varphi,(id_X \boxtimes R)(\delta_{(x,y)} \boxtimes v^{**}) \right\rangle,
$$
therefore
$$
\left\langle \varphi^\wedge,u_0\right\rangle
=\left\langle \varphi,(id_X \boxtimes R)(u_0)\right\rangle,
$$
and hence
$$
\left|\left\langle \varphi^\wedge,u_0\right\rangle \right| 
\le \left\|\varphi\right\|\left\|R\right\| \alpha(u_0; X_0, E_0)
\le (1+\varepsilon) \left\|\varphi\right\| \alpha(u_0; X_0, E_0),
$$
which implies the result since $\alpha$ is finitely generated.
\end{proof}

Lipschitz cross-norms generally do not respect subspaces, but the embedding into the bidual is respected when the Lipschitz cross-norm is finitely generated.
Compare to \cite[Lemma 13.3]{Defant-Floret}.

\begin{lemma}[Embedding lemma]
If $\alpha$ is a finitely generated generic Lipschitz cross-norm, then the mapping $id_X \boxtimes \kappa_E \colon  X \widehat{\boxtimes}_\alpha E \to X\widehat{\boxtimes}_\alpha E^{**}$ is an isometry for every pointed metric space $X$ and every Banach space $E$.
\end{lemma}

\begin{proof}
As already pointed out above, the metric mapping property implies that $\alpha(u;X,E^{**})\le\alpha(u;X,E)$ for any $u\in X \boxtimes E$ (in an abuse of notation, we are not writing the map $id_X \boxtimes \kappa_E$).
 Now, by the extension lemma,
 \begin{align*}
 \alpha(u; X, E) &= \sup\left\{\left|\left\langle \varphi,u\right\rangle\right|
 \colon \varphi \in (X \widehat{\boxtimes}_\alpha E)^*, \left\|\varphi\right\|_{(X \widehat{\boxtimes}_\alpha E)^*} \le 1 \right\} \\
 &= \sup \left\{\left|\left\langle \varphi^\wedge, u\right\rangle\right| \colon \varphi \in (X \widehat{\boxtimes}_\alpha E)^*, \left\|\varphi\right\|_{(X \widehat{\boxtimes}_\alpha E)^*} \le 1\right\} \\
 &\leq \sup \left\{\left|\left\langle \psi,u\right\rangle\right| \colon \psi \in (X \widehat{\boxtimes}_\alpha E^{**})^*, \left\|\psi\right\|_{(X \widehat{\boxtimes}_\alpha E^{**})^*} \le 1  \right\} \\
 &= \alpha(u; X, E^{**}),
 \end{align*}
giving the reverse inequality. 
\end{proof}


We are now ready to present the main result of this section.
Modulo technical assumptions,
philosophically it is a combination of Theorems \ref{teo-caracter-ideal} and \ref{teo-carac-dual-Lipschitz operator Banach space}:
Lipschitz cross-norms that are both uniform and dualizable give rise to a very satisfactory duality theory.

\subsection{The representation theorem}

\begin{theorem}
Let $A$ be a maximal generic Lipschitz operator Banach ideal and $\alpha$ a finitely generated generic Lipschitz cross-norm which are associated with each other. Then, for every pointed metric space $X$ and every Banach space $E$, the relations
\begin{align}
A(X,E^*) &= (X\widehat{\boxtimes}_\alpha E)^* \label{eqn-representation-thm-1},\\
A(X,E)   &= (X\widehat{\boxtimes}_\alpha E^*)^* \cap \Lipo(X,E) \label{eqn-representation-thm-2}
\end{align}
hold isometrically.
\end{theorem}

\begin{proof}
First, take a look at the diagram
$$
\xymatrix{
\varphi \in (X \widehat{\boxtimes}_\alpha E)^* \ar@{^{(}->}[r] \ar@{^{(}->}[d] &(X \widehat{\boxtimes}_\pi E)^* \ar@{^{(}->}[d] \ar@{=}[r] & \Lipo(X,E^*)  \\
\varphi^\wedge \in (X \widehat{\boxtimes}_\alpha E^{**})^* \ar@{^{(}->}[r] &(X \widehat{\boxtimes}_\pi E^{**})^*. & \\
}	
$$
The vertical arrows are isometries thanks to the embedding lemma, whereas the horizontal arrows are continuous because $\alpha \le \pi$.
By the extension lemma, \eqref{eqn-representation-thm-1} will follow from \eqref{eqn-representation-thm-2}.

In order to prove \eqref{eqn-representation-thm-2}, we need to show that for $f \in\Lipo(X,E)$, $f$ belongs to $A(X,E)$ if and only if the associated linear map $\varphi_f \colon  X \boxtimes_\alpha E^{*} \to \K$ is continuous, that is, there is $C>0$ such that
\begin{equation}\label{eqn-characterization-by-duality}
 |u(f)| \le C \alpha(u; X, E^*),\quad \forall u \in X \boxtimes E^*.
\end{equation}
Since $A$ is maximal, it is clear that $f \in A(X,E)$ with $\left\|f\right\|_A \le C$ if and only if
\begin{equation}\label{eqn-characterization-by-maximality} 
\left\| Q^E_L \circ f \circ I^X_Y \right\|_A \le C, \quad \forall Y \in\mfin(X),\; \forall L \in\cofin(E).
\end{equation}
Denote by $L^0$ the annihilator of $L$. Since $A(Y,E/L) = \big(Y \boxtimes_\alpha (E/L)^* \big)^* = (Y \boxtimes_\alpha L^0)^*$, and noticing that $L^0$ varies over all spaces in $\fin(E^*)$ when $L$ varies over all spaces in $\cofin(E)$, \eqref{eqn-characterization-by-maximality}
is equivalent to 
\begin{equation}\label{eqn-characterization-by-maximality-II} 
\left|u(Q^E_L \circ f \circ I^X_Y) \right|
\le C \alpha(u; Y, L^0),\quad \forall u\in Y \boxtimes L^0,
\end{equation}
whenever $Y\in\mfin(X)$ and $L^0 \in\fin(E^*)$.
Now, for such an $u\in Y \boxtimes L^0$, note that since both $I^X_Y$ and $(Q^E_L)^*$ are canonical injections,
$$
u(Q^E_L \circ f \circ I^X_Y)  = \big( (Q^E_L)^* u\big)(f \circ I^X_Y)  = u(f).
$$
Therefore \eqref{eqn-characterization-by-maximality-II} is equivalent to \eqref{eqn-characterization-by-duality} because $\alpha$ is finitely generated,
finishing the proof.
\end{proof}

We can now show that maximal generic Lipschitz operator Banach ideals can be thought of as those arising as $\Lip_\alpha$ for a finitely generated generic Lipschitz cross-norm $\alpha$.

\begin{corollary}\label{cor-maximal-iff-finitely-generated}
Let $A$ be a generic Lipschitz operator Banach ideal.
Then $A$ is maximal if and only if there exists a finitely generated generic Lipschitz cross-norm $\alpha$ such that, for every pointed metric space $X$ and every Banach space $E$,
\begin{equation}\label{eqn-corollary-representation-thm-1}
A(X,E) = (X \widehat{\boxtimes}_\alpha E^*)^* \cap \Lipo(X,E) 
\end{equation}
holds isometrically.
In this case,
\begin{equation}\label{eqn-corollary-representation-thm-2}
A(X,E^*) = (X \widehat{\boxtimes}_\alpha E)^* 
\end{equation}
also holds isometrically for every pointed metric space $X$ and every Banach space $E$.
\end{corollary}

\begin{proof}
Suppose that $A$ is maximal. Let $\alpha$ be the finitely generated generic Lipschitz cross-norm associated to $A$ given by Lemma \ref{from-ideal-to-cross-norm}.
By the representation theorem, \eqref{eqn-corollary-representation-thm-1} holds isometrically.

Now suppose that there is a finitely generated generic Lipschitz cross-norm $\alpha$ such that \eqref{eqn-corollary-representation-thm-1} holds isometrically.
It follows from the proof of the representation theorem that \eqref{eqn-corollary-representation-thm-2} must also hold isometrically, so in particular $A \sim \alpha$.
Let $f \in \Lipo(X,E)$.
If $f \in A(X,E)$ then, by the ideal property and the definition of $A^{\max}$, it follows that $\left\|f\right\|_{A^{\max}} \le \left\|f\right\|_A$.
Now assume that $f \in A^{\max}(X,E)$ with $\left\|f\right\|_{A^{\max}} \le c$.
By definition of $A^{\max}$, this means that \eqref{eqn-characterization-by-maximality} holds.
Following the proof of the representation theorem this in turn implies \eqref{eqn-characterization-by-duality}, which means that $f \in A(X,E)$ with $\left\|f\right\|_A \le c$ because of \eqref{eqn-corollary-representation-thm-1}.
\end{proof}

Another consequence is that a maximal generic Lipschitz operator Banach ideal respects the canonical embeddings into the bidual. 

\begin{corollary}\label{cor-maximal-implies-regular}
A maximal generic Lipschitz operator Banach ideal $A$ is \emph{regular}, which means: for every pointed metric space $X$, every Banach space $E$ and every $f\in \Lipo(X,E)$, $f \in A(X, E)$ if and only if $\kappa_E \circ f \in A(X,E^{**})$; moreover
$$
\left\|f \colon  X \to E \right\|_A =\left\| \kappa_E \circ f \colon  X \to E^{**} \right\|_A.
$$
\end{corollary}

\begin{proof}
Let $\alpha$ be the finitely generated generic Lipschitz cross-norm given by Corollary \ref{cor-maximal-iff-finitely-generated}.
Notice that then
$$
A(X,E) \to (X \widehat{\boxtimes}_\alpha E^*)^* = A(X, E^{**}),
$$
where the arrow is an isometry.
The desired result follows.
\end{proof}

\subsection{Lipschitz operator ideals between metric spaces}

Some important classes of Lipschitz maps satisfying an ideal property, like the Lipschitz $p$-summing maps or the maps admitting a Lipschitz factorization through a subset of an $L_p$ space, are actually defined for maps between metric spaces.
Thus, it might seem that we are losing something by insisting on having a Banach space as a codomain as it has been done so far in this paper and in the previous works \cite{ccjv,cd11,CD-Lipschitz-factorization}.
Nevertheless, we show next that generic Lipschitz operator Banach ideals satisfying a slightly stronger ideal property can be canonically extended to an ideal of Lipschitz maps between metric spaces.
Recall that $\F(X)$ denotes the Lipschitz-free Banach space of a pointed metric space $X$, and $\delta_X \colon  X \to \F(X)$ the canonical embedding.
For a Banach space $E$, the barycentric map $\beta_E \colon  \F(E) \to E$ is a norm one linear operator with $\beta_E \circ \delta_E = id_E$ (see \cite{Godefroy-Kalton}).

We will say that a generic Lipschitz operator Banach ideal $A$ is \emph{strong} if
whenever $X,Z$ are pointed metric spaces, $E$ and $F$ are Banach spaces,
$f \in A(X, E)$, $h \in \Lipo(Z, X)$ and $g \in \Lipo(E, F)$, then the composition $g f h$ belongs
to $A(Z, F)$ and $\left\|gfh\right\|_A \le \Lip(g)\left\|f\right\|_A\Lip(h)$.
The ideals $\Lipo$, $\Pi_p^L$ and $\Gamma_p^L$ are examples of strong generic Lipschitz Banach ideals.
The following proposition characterizes strong generic Lipschitz operator Banach ideals.

\begin{proposition}\label{prop-characterization-strong-ideals}
Let $A$ be a generic Lipschitz operator Banach ideal. Then $A$ is strong if and only if for every pointed metric space $X$, every Banach space $E$ and every map $f \in \Lipo(X,E)$, $f \in A(X,E)$ if and only if $\delta_E \circ f \in A(X,\F(E))$, and with $\| f \|_A = \|\delta_E \circ f\|_A$.
\end{proposition}

\begin{proof}
Suppose that $A$ is strong. Let $X$, $E$ and $f$ be as above.
If $f \in A(X,E)$, then by the ideal property $\delta_E \circ f \in A(X,\F(E))$ and 
$$
\|\ \delta_E \circ f\|_A	\le \Lip(\delta_E)\|f\|_A = \|f\|_A = \| \beta_E \circ \delta_E \circ f \|_A \le \Lip(\beta_E) \|\ \delta_E \circ f\|_A = \|\ \delta_E \circ f\|_A,
$$
so $\| f \|_A = \|\delta_E \circ f\|_A$. The same chain of inequalities shows that if $\delta_E \circ f$ is in $A$, so is $f$ and with the same norm.

For the converse implication, let  $X,Z$ be pointed metric spaces, $E$ and $F$ be Banach spaces,
$f \in A(X, E)$, $h \in \Lipo(Z, X)$ and $g \in \Lipo(E, F)$.
By \cite[Lemma 3.1]{k04}, there exists a unique bounded linear operator $\widehat{g} \colon  \F(E) \to\F(F)$ such that $\widehat{g} \circ \delta_E = \delta_F \circ g$. Furthermore, $\left\|\widehat{g}\right\| = \Lip(g)$.
Since $f \in A(X,E)$, then by hypothesis $\delta_E \circ f \in A(X, \F(E))$ and thus by the ideal property $\delta_E \circ f \circ h \in A(Z, \F(E))$.
Using the ideal property of generic Lipschitz operator Banach ideals again, we have that $ \widehat{g} \circ \delta_E \circ f \circ h = \delta_F \circ g \circ f \circ h \in A(Z,\F(F))$.
By the hypothesis, we get that $g \circ f \circ h \in A(Z,F)$. Moreover,
$$
\left\| g \circ f \circ h \right\|_A = 
\left\| \delta_F \circ g \circ f \circ h \right\|_A = 
\left\| \widehat{g} \circ \delta_E \circ f \circ h \right\|_A \le 
\| \widehat{g} \| \left\| \delta_E \circ f \circ h  \right\|_A =
\Lip(g) \left\| f \circ h \right\|_A \le 
\Lip(g) \left\| f \right\|_A \Lip(h). 
$$
\end{proof}

In the next result we define an extension of the notion of Lipschitz operator Banach ideal, now having a metric space as a codomain for the maps.
The arguments are almost the same as those used to prove Proposition \ref{prop-characterization-strong-ideals}.

\begin{proposition}\label{prop-ideals-for-metric-spaces}
Let $A$ be a strong generic Lipschitz operator Banach ideal.
For any pointed metric spaces $X$ and $Y$ and $f \in \Lipo(X,Y)$, define $f \in \tilde{A}(X,Y)$ if and only if $\delta_Y \circ f \in A(X,\F(Y))$, and denote $\left\|f\right\|_{\tilde{A}} = \left\|\delta_Y \circ f\right\|_A$.
\begin{enumerate}[(i)]
 \item For any pointed metric space $X$ and any Banach space $E$, $f \in A(X,E)$ if and only if $f \in \tilde{A}(X,E)$, and moreover $\left\|f\right\|_{\tilde{A}} = \left\|f\right\|_A$.
 \item If $f \in \tilde{A}(X, Y)$, $h \in \Lipo(W, X)$ and $g \in \Lipo(Y, Z)$, then the composition $g \circ f \circ h$ belongs
to $\tilde{A}(W, Z)$ and $\left\| g\circ f\circ h \right\|_{\tilde{A}} \le \Lip(g) \left\|f\right\|_{\tilde{A}} \Lip(h)$.
\end{enumerate}
\end{proposition}

\begin{proof}
\begin{enumerate}[(i)]
\item
If $f \in A(X,E)$, then $\delta_E \circ f \in A(X,\F(E))$ and $\left\|\delta_E \circ f\right\|_A \le \Lip(\delta_E) \left\|f\right\|_A = \left\|f\right\|_A$ by the ideal property.
Now assume that $\delta_E \circ f \in A(X,\F(E))$.
Note that $\beta_E \circ \delta_E \circ f = f$,
so, by the ideal property, $f \in A(X,E)$ and
$$
\left\|f\right\|_A
= \left\| \beta_E \circ \delta_E \circ f \right\|_A
\le \left\|\beta_E\right\| \left\| \delta_E \circ f \right\|_A
 = \left\|f\right\|_{\tilde{A}}.
$$
\item
By \cite[Lemma 3.1]{k04}, there exists a unique bounded linear operator $\widehat{g} \colon  \F(Y) \to\F(Z)$ such that $\widehat{g} \circ \delta_Y = \delta_Z \circ g$. Furthermore, $\left\|\widehat{g}\right\| = \Lip(g)$. By the ideal property, $\delta_Z \circ g \circ f \circ h\in A(W,\F(Z))$ and 
$$
\left\| g \circ f \circ h \right\|_{\tilde{A}}
=\left\| \delta_Z \circ ( g f) \circ h \right\|_A
=\left\| \widehat{g} \circ (\delta_Y \circ f) \circ h \right\|_A
\le \left\|\widehat{g}\right\|\left\|\delta_Y \circ f\right\|_A  \Lip(h)
= \Lip(g) \left\|f\right\|_{\tilde{A}} \Lip(h).
$$
\end{enumerate}
\end{proof}

In an abuse of notation, given a strong generic Lipschitz operator Banach ideal, we will still denote by $A$ its extension to metric spaces (instead of $\tilde{A}$). Note that even though we keep the notation $\left\|f\right\|_A$, when we leave the Banach space context this is no longer a norm since we lose the vector space structure. Nevertheless, it still denotes a quantitative property of the map $f$.

The following result is interesting because it characterizes a nonlinear property in terms of a linear one, closely related to \cite[Theorem 4.6]{cd11} and \cite[Theorem 4.4]{CD-Lipschitz-mixing}. Of course, as it always happens in this kind of situation, we have simplified the mapping but made the spaces more complicated.
Compare to \cite[Theorem 17.15]{Defant-Floret}, 

\begin{theorem}\label{thm-linear-characterization-of-nonlinear}
Let $A$ be a strong and maximal generic Lipschitz operator Banach ideal, and $\alpha$ the finitely generated generic Lipschitz cross-norm which is associated to $A$.
For any pointed metric spaces $X$ and $Y$, and $f \in \Lipo(X,Y)$, the following are equivalent:
\begin{enumerate}[(i)]
 \item $f \in A(X,Y)$.
 \item For all Banach spaces $G$ (or only $G = Y^\#$), $ f \boxtimes id_G \colon  X\widehat{\boxtimes}_\alpha G \to Y\widehat{\boxtimes}_\pi G$ is continuous.
\end{enumerate}
In this case,
$$
\left\|f\right\|_A = \left\| f \boxtimes id_{Y^\#} \colon  X \widehat{\boxtimes}_\alpha Y^\# \to Y \widehat{\boxtimes}_\pi Y^\# \right\| \ge \left\| f \boxtimes id_G \colon  X\widehat{\boxtimes}_\alpha G \to Y\widehat{\boxtimes}_\pi G \right\|.
$$
\end{theorem}

\begin{proof}
Suppose that $f \in A(X,Y)$, and let $G$ be a Banach space. The boundedness of
$$
f \boxtimes id_G \colon  X\widehat{\boxtimes}_\alpha G \to Y\widehat{\boxtimes}_\pi G
$$
will follow from the boundedness of the adjoint map
$$
(f \boxtimes id_G)^* \colon  (Y\widehat{\boxtimes}_\pi G)^* = \Lipo(Y,G^*) \to (X\widehat{\boxtimes}_\alpha G)^* = A(X,G^*).
$$
Now, for $v\in G$, $x\in X$ and $h\in\Lipo(Y,G^*)$,
$$
\pair{[(f \boxtimes id_G)^*h](x)}{v}
= \big( \delta_{(x,0)} \boxtimes v \big) \big( (f \boxtimes id_G)^*h \big)
= \big( (f \boxtimes id_G)\big[\delta_{(x,0)} \boxtimes v \big] \big) \big( h \big)
= \big( \delta_{(f(x),0)} \boxtimes v \big) \big( h \big)
= \pair{h(f(x))}{v}.
$$
Therefore $(f \boxtimes id_G)^*$ is
given by $h \in \Lipo(Y,G^*) \mapsto h \circ f \in A(X,G^*)$,
which has norm at most $\left\|f\right\|_A$ because of the ideal property.

Now suppose that $f \boxtimes id_{Y^\#} \colon  X\widehat{\boxtimes}_\alpha Y^\# \to Y\widehat{\boxtimes}_\pi Y^\#$ has norm $c$.
By definition, $f \in A(X,Y)$ if and only if $\delta_Y \circ f \in A(X,\F(Y))$ and with the same norm,
which by the representation theorem is equivalent to having $\delta_Y \circ f$ define an element of $(X \widehat{\boxtimes}_\alpha \F(Y)^*)^* = (X\widehat{\boxtimes}_\alpha Y^{\#})^*$.
Therefore, we seek to prove that given $u \in X\widehat{\boxtimes} Y^\#$,
$|u(\delta_Y \circ f)| \le c \alpha(u)$.
Note that for a
given $u \in X\widehat{\boxtimes} Y^\#$, 
$u(f \boxtimes id_{Y^\#})$ belongs to $Y\widehat{\boxtimes} Y^\#$.
Since $\kappa_{\F(Y)} \circ \delta_Y : Y \to Y^{\#*}$, we may consider
$
\big[ (f \boxtimes id_{Y^\#})u \big](\kappa_{\F(Y)} \circ \delta_Y).
$
Note that this is in fact just $u(f)$, since the maps $\delta_Y$, $\kappa_{\F(Y)}$ and $id_{Y^\#}$ are inclusions.
Therefore,
$$
|u(f)| = \big|\big[ (f \boxtimes id_{Y^\#})u \big](\kappa_{\F(Y)} \circ \delta_Y)\big|
\le \Lip( \kappa_{\F(Y)} \circ \delta_Y ) \pi((f \boxtimes id_{Y^\#})u )
\le c \alpha(u)
$$
and the conclusion follows.
\end{proof}

\begin{remark}
 It should be pointed out that in Theorem \ref{thm-linear-characterization-of-nonlinear}, small adaptations of the proof show that when the codomain is a Banach space $E$ (respectively, $F^*$) in part (ii) it suffices to consider $G=E^*$ (respectively, $G=F$).
\end{remark}

\bibliographystyle{amsplain}

\end{document}